\documentclass[12pt,a4paper]{amsart}
\usepackage{amsfonts,color}
\usepackage{amsthm}
\usepackage{amsmath}
\usepackage{amscd}
\usepackage[latin2]{inputenc}
\usepackage{t1enc}
\usepackage[mathscr]{eucal}
\usepackage{indentfirst}
\usepackage{graphicx}
\usepackage{graphics}
\usepackage{pict2e}
\usepackage{epic}
\usepackage{url}
\usepackage{epstopdf}
\usepackage{comment}
\usepackage{todonotes}
\usepackage{stmaryrd}
\usepackage{amssymb}

\allowdisplaybreaks

\numberwithin{equation}{section}
\usepackage[margin=2.2cm]{geometry}

\usepackage{pgfplots}
\usepackage{xcolor}
\usepackage{tikz}
\usetikzlibrary{matrix,arrows,decorations.pathmorphing}
\usetikzlibrary{calc,decorations.pathreplacing}
\usetikzlibrary{quotes,angles}
\usetikzlibrary{shapes}
\usetikzlibrary{patterns}

\tikzstyle{vertex}=[draw=black,circle,fill=black,minimum size=6pt, inner sep=0pt, outer sep=0pt,text=black,line width=0mm]

\tikzstyle{Sqvertex}=[draw=black,shape=rectangle, minimum size=10pt, fill=white]

\tikzstyle{Cvertex}=[draw=black,shape=circle, minimum size=6pt, fill=white]

\tikzstyle{vertex_blue}=[draw=black,circle,fill=blue,minimum size=6pt, inner sep=0pt, outer sep=0pt,text=black,line width=0mm]
\tikzstyle{vertex_red}=[draw=black,circle,fill=red,minimum size=6pt, inner sep=0pt, outer sep=0pt,text=black,line width=0mm]
\tikzstyle{vertex_green}=[draw=black,circle,fill=green,minimum size=6pt, inner sep=0pt, outer sep=0pt,text=black,line width=0mm]
\tikzstyle{c0}=[shape=circle, minimum size=4pt, fill=white]
\tikzstyle{c1}=[shape=rectangle, minimum size=7pt, fill=red]
\tikzstyle{c2}=[shape=diamond, minimum size=10pt, fill=blue]
\tikzstyle{mybox} = [rectangle, rounded corners, minimum width=3cm, minimum height=1cm,text centered, draw=black]
\tikzset{base/.style = {rectangle, rounded corners, draw=black,
                           minimum width=3cm, minimum height=1cm,
                           text centered}}

\usepgfplotslibrary{fillbetween}
\pgfplotsset{mystyle/.style={%
        xmin=-2,
        xmax=7.9,
        ymin=-1,
        xtick = {1,3},
        xticklabels = {{1},$d-1$},
        ytick = {1}
    }
}

\pgfplotsset{mystyle2/.style={%
        xmin=-2,
        xmax=7.9,
        ymin=-1,
        ymax=6,
        xtick = {1,3},
        xticklabels = {{1},$d-1$},
        ytick = {1}
    }
}

\definecolor{darkerblue}{HTML}{065A82} 
\definecolor{lighterblue}{HTML}{1C7293} 

\theoremstyle{plain}
\newtheorem{Th}{Theorem}[section]
\newtheorem{Lemma}[Th]{Lemma}
\newtheorem{Cor}[Th]{Corollary}
\newtheorem{Prop}[Th]{Proposition}

 \theoremstyle{definition}
\newtheorem{Def}[Th]{Definition}
\newtheorem{Conj}[Th]{Conjecture}
\newtheorem{Rem}[Th]{Remark}
\newtheorem{?}[Th]{Problem}

\renewcommand{\P}{\mathbb{P}}

\newcommand{\cI}{\mathcal I}
\newcommand{\cB}{\mathcal B}
\newcommand{\PP}{\mathbb P}

\newcommand{\eff}{\mathrm{eff}}
\newcommand{\ia}{\mathrm{ia}}
\newcommand{\ea}{\mathrm{ea}}

\newcommand{\He}{\operatorname{He}}
\newcommand{\ee}{\mathrm e}
\newcommand{\cF}{\mathcal F}
\newcommand{\cT}{\mathcal T}

\begin{document}

\title[An inequality for the number of independent sets of matroids]{An inequality for the number of independent sets of matroids with an application to the forest-tree ratio of graphs}

\author[F. Bencs]{Ferenc Bencs}

\address{Centrum Wiskunde \& Informatica, Amsterdam, The Netherlands. P.O. Box 94248 1090 GE
Amsterdam The Netherlands}
\email{ferenc.bencs@gmail.com}

\author[P. Csikv\'ari]{P\'{e}ter Csikv\'{a}ri}

\address{Alfr\'ed R\'enyi Institute of Mathematics, H-1053 Budapest, Re\'altanoda utca 13-15 \and E\"otv\"os Lor\'and University, H-1117 Budapest, P\'azm\'any P\'eter s\'et\'any 1/C }
\email{peter.csikvari@gmail.com}

\thanks{The first author is supported by the Netherlands Organisation of Scientific Research (NWO): VI.Veni.222.303. The second author is supported by the Hungarian National Research, Development and Innovation
Office, Advanced grant 153378, and by Dynasnet European Research Council Synergy project -- grant number ERC-2018-SYG 810115.
}



\begin{abstract}
Let $M=(E,\mathcal{I})$ be a matroid of rank $r$. Let $\mathcal{I}_k$ be the independent sets of size $k$, and let $I_k=|\mathcal{I}_k|$ and $I=|\mathcal{I}|$.
We show that if every set $F\in \mathcal{I}_{r-1}$ is contained in at least $\delta$ bases, then
$$\ln \left(\frac{I}{I_r}\right)\geqslant \frac{I_{r-1}}{I_r}\cdot \delta \ln \left(1+\frac{1}{\delta}\right).$$
In particular, we have 
$$\frac{I}{I_r}\geqslant 2^{I_{r-1}/I_r}.$$
By combining this result with several other ideas, we prove that if $G$ is a simple connected graph on $n$ vertices, and $F(G)$ and $T(G)$ denote its numbers of spanning forests and spanning trees, respectively, then
$$\frac{F(G)}{T(G)}\geqslant \frac{F(K_n)}{T(K_n)},$$
where $K_n$ is the complete graph on $n$ vertices. Equality holds if and only if $G=K_n$.

\end{abstract}

\maketitle

\section{Introduction}

For a connected graph $G=(V,E)$ let $F(G)$ denote the number of spanning forests of $G$, that is, edge subsets $A\subseteq E$ without a cycle, and let $T(G)$ denote the number of spanning trees of $G$.

Cayley's formula gives that for the complete graph $K_n$ we have $T(K_n)=n^{n-2}$. Since one spanning tree contains $2^{n-1}$ spanning forests as subgraphs, one might expect that $F(K_n)$ is exponentially larger than $T(K_n)$. However, this is not the case: the same spanning forest is contained in many spanning trees. R\'enyi \cite{renyi1959some} proved that $F(K_n)\sim\sqrt{e}n^{n-2}$, where $e$ is the base of the natural logarithm.  Note that $\sqrt{e}\approx 1.648$.

The following conjecture was formulated by Mohan Ravichandran and independently by Guillem Perarnau and  Guillaume Chapuy (personal communication).

Let
\begin{equation} \label{eq:Q-def}
Q(G):=\frac{F(G)}{T(G)}.
\end{equation}

\begin{Conj} \label{forest-tree-ratio-conj}
Let $G$ be a simple connected graph on $n$ vertices. Then
$$Q(G)\geqslant Q(K_n).$$
\end{Conj}

The main result of this paper is a proof of this conjecture.

\begin{Th} \label{forest-tree-ratio-thm}
Let $G$ be a simple connected graph on $n$ vertices. Then
$$Q(G)\geqslant Q(K_n).$$
Equality holds if and only if $G=K_n$.
\end{Th}

The proof of Theorem~\ref{forest-tree-ratio-thm} is decomposed into two theorems: one for graphs with at most $\binom{n}{2}-0.49n$ edges, and one for graphs with more than $\binom{n}{2}-0.49n$ edges. More precisely, one argument works for graphs with less than $\binom{n}{2}-\left(\frac{1}{4}+\varepsilon\right)n$ edges, and another argument works for graphs with more than $\binom{n}{2}-\left(\frac{1}{2}-\varepsilon\right)n$ edges. While there is plenty of room between $\frac{1}{4}$ and $\frac{1}{2}$ we will need to take care of the small $n$ cases. This is where the choice $0.49n$ will be convenient. So we will prove the following two theorems along with some more quantitative versions of them to make sure that we treat all positive integers $n$.

\begin{Th} \label{forest-tree-ratio-2}
For every $\varepsilon>0$ there exists an $n_0=n_0(\varepsilon)$ such that if $G$ has $n\geqslant n_0$ vertices and less than $\binom{n}{2}-\left(\frac{1}{4}+\varepsilon\right)n$ edges, then
$$Q(G)> Q(K_n).$$
\end{Th}

\begin{Th} \label{forest-tree-ratio-3}
For every $\varepsilon>0$ there exists an $n_0=n_0(\varepsilon)$ such that if $G$ has $n\geqslant n_0$ vertices and more than $\binom{n}{2}-\left(\frac{1}{2}-\varepsilon\right)n$ edges, then
$$Q(G)\geqslant Q(K_n)$$
with equality if and only if $G=K_n$.
\end{Th}

The splitting of Theorem~\ref{forest-tree-ratio-thm} is not just aesthetic: the arguments proving Theorems~\ref{forest-tree-ratio-2} and \ref{forest-tree-ratio-3} are completely different. Interestingly, both arguments fail spectacularly in the regime of the other theorem. The proof of Theorem~\ref{forest-tree-ratio-2} cannot work for graphs with more than $\binom{n}{2}-cn$ for some small $c$ (though $\frac{1}{4}$ could be improved with some extra work). This argument also gives the following simpler theorem.

\begin{Th} \label{forest-tree-ratio}
Let $G$ be a simple connected graph on at least $2$ vertices. Then
$Q(G)> \sqrt{e}$.
\end{Th}

The argument proving Theorem~\ref{forest-tree-ratio-3} fails for graphs with less than $\binom{n}{2}-\frac{n}{2}$ edges. This proof is based on the inclusion-exclusion formula with very subtle analysis. The core idea is a beautiful formula for the number of forests of certain contracted versions of the complete graph.

We also note that the proof of Theorem~\ref{forest-tree-ratio-2} was contributed entirely by the human authors of the article, while the proof of Theorem~\ref{forest-tree-ratio-3} was provided largely by ChatGPT 5.6 Sol. Here, human authors only reorganized the proof for better readability, clarified some details, and improved some constants. There is also a Jupyter notebook in the following github repository that aims to help the Reader with checking computational results:
\bigskip

\fbox{\begin{minipage}{0.9\textwidth}
\url{https://github.com/csikvari/forest-tree-ratio/blob/main/Forest_tree_ratio_reproducibility.ipynb}
\end{minipage}}
\bigskip

The proofs of Theorems~\ref{forest-tree-ratio-2} and \ref{forest-tree-ratio} rely on the following two theorems. One of them is a formula of Richman, Shokrieh and Wu \cite{richman2023counting}. For a brief introduction to the theory of resistances used in this paper, see the beginning of Section~\ref{forest-tree-ratio-2-sect}.

\begin{Th}[Richman, Shokrieh and Wu \cite{richman2023counting}] \label{thm: RSW-formula}
Let $G$ be a connected graph and let $F_2(G)$ denote the number of forests with two connected components. Suppose that we assign unit resistance to every edge of the graph, and let $R_{\eff}(u,v)$ be the effective resistance between vertices $u$ and $v$. Let us orient each edge arbitrarily and let $e^-$ be the tail and $e^+$ be the head of an edge. If $q$ is an arbitrary vertex, then
\begin{align} \label{RSW-formula}
\frac{F_2(G)}{T(G)}=\frac{1}{4}\sum_{e\in E}R_{\eff}(e^-,e^+)^2+\frac{3}{4}\sum_{e\in E}(R_{\eff}(q,e^-)-R_{\eff}(q,e^+))^2.
\end{align}
\end{Th}

We will refer to the identity \eqref{RSW-formula} as RSW-formula throughout the paper. It is a variant of the formula of Kassel and Wilson \cite{kassel2016looping} (see also \cite{kassel2015random}) that was obtained by building on the work of Myrvold \cite{myrvold1992counting}. As we will see, a consequence of the RSW-formula is that 
$$\frac{F_2(G)}{T(G)}\geqslant \frac{1}{2}$$
holds true for every simple connected graph $G$.
This statement is asymptotically tight for the complete graphs.
In fact, for the complete graphs $K_n$ R\'enyi \cite{renyi1959some} proved that for fixed $s$ we have
$$\frac{F_s(K_n)}{T(K_n)}\simeq \frac{1}{2^{s-1}(s-1)!},$$
where $F_s(G)$ denotes the number of forests with $s$ components. (So $F_1(G)=T(G)$.) While this immediately implies that
$$\frac{F(G)}{T(G)}\geqslant \frac{T(G)+F_2(G)}{T(G)}\geqslant \frac{3}{2},$$
in order to get the stronger, asymptotically tight bound $\sqrt{e}$ we need a theorem that connects $\frac{F(G)}{T(G)}$ with $\frac{F_2(G)}{T(G)}$. This is exactly the content of the advertised lemma, stated in the abstract, which is valid not only for graphs (graphic matroids), but for general matroids.

\begin{Th} \label{ratio-minimum}
Let $M=(E,\mathcal{I})$ be a matroid of rank $r\geqslant 1$. Let $\mathcal{I}_k$ be the independent sets of size $k$, and let $I_k=|\mathcal{I}_k|$ and $I=|\mathcal{I}|$. For a set $F\in \mathcal{I}_{r-1}$ let $\delta(F)$ be the number of bases containing $F$. Then
$$\ln \left(\frac{I}{I_r}\right)\geqslant \frac{I_{r-1}}{I_r}   
\cdot \left[\frac{1}{I_{r-1}}\sum_{F\in \mathcal{I}_{r-1}}\delta(F)\ln \left(1+\frac{1}{\delta(F)}\right)\right].$$
In particular, if every set $F\in \mathcal{I}_{r-1}$ is in at least $\delta$ bases, then
$$\ln \left(\frac{I}{I_r}\right)\geqslant \frac{I_{r-1}}{I_r}\cdot \delta \ln \left(1+\frac{1}{\delta}\right).$$
In particular, we have 
$$\frac{I}{I_r}\geqslant 2^{I_{r-1}/I_r}.$$
\end{Th}

For graphic matroids, the above theorem says that
$$\ln\left(\frac{F(G)}{T(G)}\right)\geqslant \frac{F_2(G)}{T(G)}\cdot \delta \ln \left(1+\frac{1}{\delta}\right),$$
where $\delta$ is the size of the smallest cut. Seemingly, this is not exactly what we want, because if $\delta$ is small, then $\delta \ln\left(1+\frac{1}{\delta}\right)$ is much smaller than $1$. Luckily, when $\delta$ is small, then the effective resistance between the endpoints of the edges in small cuts will give an extra contribution to $\frac{F_2(G)}{T(G)}$ in the RSW-formula. 
\bigskip

\noindent \textbf{Notation.} Throughout this paper $G$ will be a simple connected graph on $n$ vertices with $m$ edges and edge-connectivity $\delta$, which means that the size of the smallest cut is $\delta$. $T(G)$ will denote its number of spanning trees, and $F(G)$ denotes its number of spanning forests. $K_n$ is the complete graph on $n$ vertices. We use the notation $Q(G)=\frac{F(G)}{T(G)}$, $F_n=F(K_n)$, $T_n=T(K_n)$ and $Q_n=Q(K_n)$.
\bigskip

\noindent \textbf{This paper is organized as follows.} In Section
~\ref{permutation-Tutte-sect} we prove Theorem~\ref{ratio-minimum}. In Section~\ref{forest-tree-ratio-2-sect} we prove Theorems~\ref{forest-tree-ratio-2} and \ref{forest-tree-ratio}. In Section~\ref{incusion-exclusion-sect} we give the proofs of Theorems~\ref{forest-tree-ratio-3} and \ref{forest-tree-ratio-3-quantitative}. We end the paper with various conjectures and remarks on the methods used in the paper.

\section{Permutation Tutte polynomial and Theorem~\ref{ratio-minimum}} \label{permutation-Tutte-sect}

The proof of Theorem~\ref{ratio-minimum} heavily relies on the theory of the permutation Tutte polynomial. The idea is that the Tutte polynomial $T_G(x,y)$ can be written as a sum of the permutation Tutte polynomials $\widetilde{T}_{H_j}(x,y)$ for certain bipartite graphs $H_j$. As a consequence, certain inequalities valid for the permutation Tutte polynomial transfer to the Tutte polynomial.

\begin{Def}[Beke, Cs\'aji, Csikv\'ari, Pituk \cite{beke2024permutation}] \label{main-def}
Let $H=(A,B,E)$ be a bipartite graph. Suppose that $V(H)=[m]$. For a permutation $\pi:[m]\to [m]$, we say that a vertex $i\in A$ is internally active if
$$\pi(i)>\max_{j\in N_H(i)}\pi(j),$$
where the maximum over an empty set is set to be $-\infty$.
Similarly, we say that vertex $j\in B$ is externally active if
$$\pi(j)>\max_{i\in N_H(j)}\pi(i).$$
Let $\ia(\pi)$ and $\ea(\pi)$ be the number of internally and externally active vertices in $A$ and $B$, respectively.
Let
$$\widetilde{T}_H(x,y)=\frac{1}{m!}\sum_{\pi \in S_m}x^{\ia(\pi)}y^{\ea(\pi)},$$
where $S_m$ denotes the set of all permutations on $m$ elements.
We will call $\widetilde{T}_H(x,y)$ the permutation Tutte polynomial of $H$.
\end{Def}

The above definition is motivated by the following theorem of Tutte.

\begin{Th}[Tutte \cite{tutte1954contribution}] \label{ia-ea-characterization}
Let $G$ be a connected graph with $m$ edges. Label the edges with $1,2,\dots,m$ arbitrarily. In the case of a spanning tree $T$ of $G$, let us call an edge $e\in E(T)$ internally active if $e$ has the largest label among the edges $f'\in E(G)$ joining the connected components of the graph $T-e$. Let us call an edge $f\notin E(T)$ externally active if $f$ has the largest label among the edges in the the fundamental cycle created by adding $f$ to $T$. Let $\mathrm{ia}(T)$ and $\mathrm{ea}(T)$ be the number of internally and externally active edges, respectively. Then
$$T_G(x,y)=\sum_{T\in \mathcal{T}(G)}x^{\mathrm{ia}(T)}y^{\mathrm{ea}(T)},$$
where the summation runs over all spanning trees of $G$.
\end{Th}

Tutte originally used this formula as a definition of the Tutte polynomial \cite{tutte1954contribution}. This theorem or definition naturally extends to matroids even though Tutte's original paper only concerned graphs. This characterization of the Tutte polynomial immediately shows that the coefficients of the Tutte polynomial are non-negative. In this theorem, the same edge labelling must be used for every spanning tree. For those who have never seen this definition before, it might be very surprising that the Tutte polynomial is independent of the actual choice of the labeling.

To explain the connection between $T_G(x,y)$ and $\widetilde{T}_H(x,y)$, we need the concept of the local basis exchange graph.

\begin{Def} The local basis exchange graph $H[T]$ of a graph $G=(V,E)$ with respect to a spanning tree  $T$ is defined as follows.
 The graph $H[T]$ is a bipartite graph whose vertices are the edges of $G$. One bipartite class consists of the edges of $T$, the other consists of the edges of $E\setminus T$, and we connect a spanning tree edge $e$ with a non-tree edge $f$ if $f$ is in the cut determined by $e$ and $T$, equivalently, $e$ is in the cycle determined by $f$ and $T$. This is also equivalent to $T-e+f$ being a spanning tree of $G$ again.
 
 Clearly, this definition works for general matroids $M=(E,\mathcal{I})$ and their bases. 
 If $A$ is a basis, then let $H[A]$ be the bipartite graph with $V(H[A])=E$,  where one part consists of the elements of $A$, the other part consists of $E\setminus A$, and $e\in A$ and $f\in E\setminus A$ are adjacent in the bipartite graph $H[A]$ if $A-e+f$ is again a basis.
\end{Def}
Figure 1 shows a graph $G$ with a spanning tree $T$ and the bipartite graph $H[T]$ obtained from $T$. 
\bigskip

For a fixed labeling of the edges of $G$, we get a labeling of the vertices of $H[T]$, and the internally (externally) active edges of $G$ correspond to internally (externally) active vertices of $H[T]$, so the two definitions of internal and external activity are compatible. The following lemma connects the Tutte polynomial with the permutation Tutte polynomial.

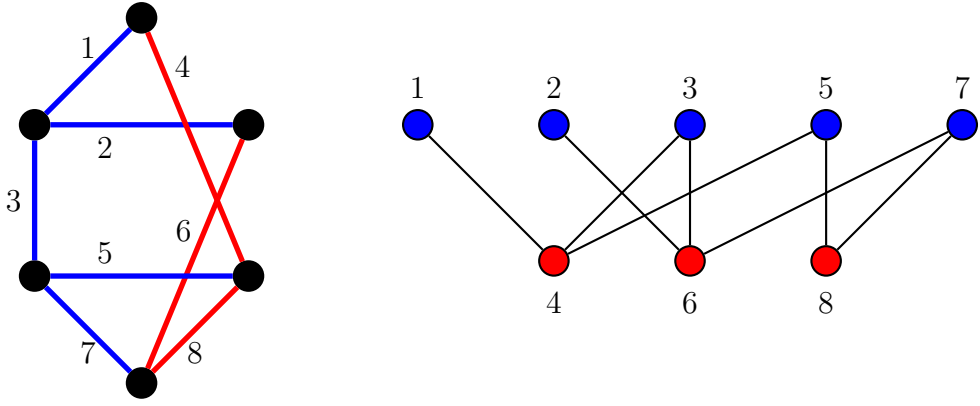
\begin{figure}[htp] 
\begin{tikzpicture}[, scale=0.33, baseline=0pt, node distance={20mm}, thick, main/.style = {draw, circle, fill=black}] 
\node[main] (1) {}; 
\node[main] (2) [above right of=1] {}; 
\node[main] (3) [below right of=2]{}; 
\node[main] (4) [below of=1]{}; 
\node[main] (5) [below of=3]{}; 
\node[main] (6) [below right of=4]{}; 
\draw [color=blue,line width=2pt](1) edge node[
above,black]{$1$} (2) ; 
\draw [color=blue, line width=2pt](1) edge node[pos=0.3, below, black]{$2$} (3) ; 
\draw [color=blue, line width=2pt](1) edge node[left, black]{$3$} (4) ; 
\draw [color=red, line width=2pt](2) edge  node[pos=0.15, right, black]{$4$} (5) ; 
\draw [color=red, line width=2pt](3) edge  node[pos=0.4, left, black]{$6$} (6) ; 
\draw [color=blue, line width=2pt](4) edge node[pos=0.3, above, black]{$5$} (5) ; 
\draw [color=blue, line width=2pt](4) edge  node[below, black]{$7$} (6) ; 
\draw [color=red, line width=2pt](5) edge  node[below, black]{$8$} (6) ; 
\end{tikzpicture} 
\qquad \qquad
\begin{tikzpicture}[, scale=0.33, baseline=0pt, node distance={18mm}, thick, main/.style = {draw, circle}]
\node[main, fill=blue, label=$1$] (1) {}; 
\node[main, fill=blue, label=$2$] (2) [right of=1]{};
\node[main,fill=blue, label=$3$] (3) [right of=2]{}; 
\node[main,fill=blue, label=$5$] (4) [right of=3]{}; 
\node[main,fill=blue, label=$7$] (5) [right of=4]{}; 
\node[main,fill=red, label={[yshift=-30pt]$4$}] (6) [below of=2]{}; 
\node[main,fill=red, label={[yshift=-30pt]$6$}] (7) [below of=3]{};
\node[main,fill=red, label={[yshift=-30pt]$8$}] (8) [below of=4]{};
\draw (1) -- (6) ; 
\draw (2) -- (7) ; 
\draw (3) -- (6) ; 
\draw (3) -- (7) ; 
\draw (4) -- (6) ; 
\draw (4) -- (8) ; 
\draw (5) -- (7) ; 
\draw (5) -- (8) ; 
\end{tikzpicture} 
\caption{An example for a graph $G$ and the local basis exchange graph $H[T]$ obtained from a spanning tree $T$.}
\end{figure}

\begin{Lemma}[Beke, Cs\'aji, Csikv\'ari, Pituk \cite{beke2024permutation}] \label{conn}
Let $M$ be a matroid. For each basis $A$ of $M$, let $H[A]$ be the local basis exchange graph with respect to $A$. Then
$$T_M(x,y)=\sum_{A\in \mathcal{B}(M)}\widetilde{T}_{H[A]}(x,y),$$
where the sum runs over the set of bases $\mathcal{B}(M)$ of $M$.
\end{Lemma}

We use the following bound on the permutation Tutte polynomial  proved in \cite{beke2024permutation}, which is a simple consequence of Harris's  inequality \cite{harris1960lower}.

\begin{Th}[Beke, Cs\'aji, Csikv\'ari, Pituk \cite{beke2024permutation}] \label{lower-bound}
Let $H=(A,B,E)$ be an arbitrary bipartite graph, and let $d_i$ be the degree of a vertex $i$. Suppose that $0\leqslant x\leqslant 1$ and $y\geqslant 1$, or $0\leqslant   y\leqslant 1$ and $x\geqslant 1$. Then
$$\widetilde{T}_H(x,y)\geqslant \prod_{i\in A}\left(1+\frac{x-1}{d_i+1}\right) \cdot \prod_{j\in B}\left(1+\frac{y-1}{d_j+1}\right).$$
\end{Th}

Now we are ready to prove Theorem~\ref{ratio-minimum}.

\begin{proof}[Proof of Theorem~\ref{ratio-minimum}]
We have
\begin{align*}
\frac{I}{I_r} &= \frac{T_M(2,1)}{I_r}\\
&= \frac{1}{I_r}\sum_{A\in \mathcal{I}_r}\widetilde{T}_{H[A]}(2,1)\\
&\geqslant \frac{1}{I_r}\sum_{A\in \mathcal{I}_r}\prod_{e\in A}\left(1+\frac{1}{d_{H[A]}(e)+1}\right)\\
&= \frac{1}{I_r}\sum_{A\in \mathcal{I}_r}\prod_{e\in A}\left(1+\frac{1}{\delta(A-e)}\right)\\
&\geqslant \left(\prod_{A\in \mathcal{I}_r}\prod_{e\in A}\left(1+\frac{1}{\delta(A-e)}\right)\right)^{1/I_r}\\
&= \left(\prod_{I\in \mathcal{I}_{r-1}}\left(1+\frac{1}{\delta(I)}\right)^{\delta(I)}\right)^{1/I_r}
\end{align*}
Taking logarithms we get that
$$\ln\left(\frac{I}{I_r}\right)\geqslant \frac{I_{r-1}}{I_r}\cdot \frac{1}{I_{r-1}}\sum_{I\in \mathcal{I}_{r-1}}\delta(I)\ln \left(1+\frac{1}{\delta(I)}\right).$$
Since the function $\delta \to \delta \ln\left(1+\frac{1}{\delta}\right)$ is monotone increasing, we obtain a  further lower bound by replacing the average $\frac{1}{I_{r-1}}\sum_{I\in \mathcal{I}_{r-1}}\delta(I)\ln \left(1+\frac{1}{\delta(I)}\right)$ with the minimum $\delta \ln\left(1+\frac{1}{\delta}\right)$.
\end{proof}

\subsection{Second proof of Theorem~\ref{ratio-minimum}} In this section, we give another proof of
Theorem~\ref{ratio-minimum}. This proof was found by ChatGPT 5.5 Pro. It relies on the bijective basis exchange theorem instead of the permutation Tutte polynomial. 

\begin{proof}[Second proof of Theorem~\ref{ratio-minimum}]

Let $\cB=\cI_r$, and for an independent set $A\in\cI$, define
\[
 \Delta(A):=|\{B\in\cB:A\subseteq B\}|.
\]
Thus if $S\in\cI$ and $|S|=r-1$, then $\Delta(S)$ is the number of bases containing $S$. In this section write
\[
 \delta(S):=\Delta(S)
 \qquad (S\in\cI,\ |S|=r-1),
\]
and
\[
 \delta:=\min\{\delta(S):S\in\cI,\ |S|=r-1\}.
\]
Since every independent set extends to a basis, $\delta\geq 1$.

We start from the exact identity
\begin{equation} \label{eq: 3}
 \frac{I}{I_r}
 =
 \frac1{I_r}
 \sum_{B\in\cB}\sum_{A\subseteq B}\frac{1}{\Delta(A)}. 
\end{equation}
Indeed, each independent set $A$ is contained in exactly $\Delta(A)$ bases, and hence its total contribution to the right-hand side is $1/I_r$ times $\Delta(A)\cdot 1/\Delta(A)$.

Fix a basis $B$ and an independent set $A\subseteq B$. We claim that
\begin{equation} \label{eq: 4}
 \Delta(A)
 \leq
 \prod_{e\in B\setminus A}\Delta(B\setminus\{e\}). 
\end{equation}
To prove this, contract $A$. In $M/A$, the set $C=B\setminus A$ is a basis. Bases of $M/A$ correspond exactly to bases of $M$ containing $A$, so their number is $\Delta(A)$.

For a basis $D$ of $M/A$, the bijective basis-exchange theorem \cite{brualdi1969comments} gives a bijection
\[
 \varphi_D:C\to D
\]
such that
\[
 C-e+\varphi_D(e)
\]
is a basis of $M/A$ for every $e\in C$. Choose one such bijection for each $D$. Then the map
\[
 D\longmapsto (\varphi_D(e))_{e\in C}
\]
is injective, because the tuple determines $D$ as its set of entries. For fixed $e\in C$, the number of possible values of $\varphi_D(e)$ is exactly
\[
 \Delta(B\setminus\{e\}),
\]
since $C-e+x$ is a basis of $M/A$ precisely when $B-e+x$ is a basis of $M$, i.e. precisely when that basis contains $B\setminus\{e\}$. This proves \eqref{eq: 4}.

Taking reciprocals in (\ref{eq: 4}) gives, for every $A\subseteq B$,
$$
 \frac{1}{\Delta(A)}
 \geqslant
 \prod_{e\in B\setminus A}\frac1{\Delta(B\setminus\{e\})}.
$$
Therefore
\begin{equation} \label{eq: 5}
 \sum_{A\subseteq B}\frac{1}{\Delta(A)}
 \geqslant
 \sum_{A\subseteq B}
 \prod_{e\in B\setminus A}\frac1{\Delta(B\setminus\{e\})}
 =
 \prod_{e\in B}\left(1+\frac1{\Delta(B\setminus\{e\})}\right). 
\end{equation}
Combining (\ref{eq: 3}) and (\ref{eq: 5}),
\[
 \frac{I}{I_r}
 \geqslant
 \frac{1}{I_r}
 \sum_{B\in\cB}
 \prod_{e\in B}\left(1+\frac{1}{\Delta(B\setminus\{e\})}\right).
\]
Since $\ln$ is concave, Jensen's inequality implies
\begin{equation} \label{eq: 6}
 \ln\left(\frac{I}{I_r}\right)
 \geq
 \frac1{I_r}
 \sum_{B\in\cB}
 \sum_{e\in B}
 \ln\left(1+\frac1{\Delta(B\setminus\{e\})}\right). 
\end{equation}

Now group the double sum in \eqref{eq: 6} according to the independent $(r-1)$-set
\[
 S=B\setminus\{e\}.
\]
For a fixed $S\in\cI$ with $|S|=r-1$, there are exactly $\delta(S)$ bases containing $S$, and for each such basis $B$ there is a unique element $e\in B$ with $S=B\setminus\{e\}$. Hence \eqref{eq: 6} becomes
\[
 \ln\left(\frac{I}{I_r}\right)
 \geq
 \frac{1}{I_r}
 \sum_{\substack{S\in\cI\\ |S|=r-1}}
 \delta(S)\ln\left(1+\frac1{\delta(S)}\right).
\]
This proves the stronger inequality.

Finally, the function
\[
 f(x)=x\ln\left(1+\frac1x\right)
\]
is increasing for $x>0$; hence $\delta(S)\geq\delta$ implies
\[
 \delta(S)\ln\left(1+\frac1{\delta(S)}\right)
 \geq
 \delta\ln\left(1+\frac1\delta\right).
\]
There are $I_{r-1}$ independent sets of size $r-1$, so the desired consequence follows.
\end{proof}

\section{Proofs of Theorems~\ref{forest-tree-ratio-2} and \ref{forest-tree-ratio}.} \label{forest-tree-ratio-2-sect}

In this section, we prove Theorems~\ref{forest-tree-ratio-2} and \ref{forest-tree-ratio} together with the following quantitative version of Theorem~\ref{forest-tree-ratio-2}.

\begin{Th} \label{forest-tree-ratio-2-quantitative}
Let $n\geqslant 8$. Let $G$ be a connected graph on $n$ vertices with $m=\binom{n}{2}-h$ edges, where $h\geqslant \lfloor 0.49n+1/2\rfloor$. Then $Q(G)>Q(K_n)$.
\end{Th}

The main reason why we choose $\lfloor 0.49n+1/2\rfloor$ is that in Section~\ref{incusion-exclusion-sect} we will show that if $h\leqslant 0.49n$, then it is again true that $Q(G)\geqslant Q(K_n)$. Together, the two parts imply that $Q(G)\geqslant Q_n$ for every $n\geqslant 1$.
\bigskip

To prove our theorems we will heavily rely  on the formula~\eqref{RSW-formula} that expresses the ratio $\frac{F_2(G)}{T(G)}$ in terms of effective resistances. We will only use very basic facts about electrical networks. For a thorough introduction to the topic,  we recommend Chapters 2 and 9 of Bollob\'as's book \cite{bollobas1998modern}.  

An electrical network is just a graph $G$ with a function $r:E\to \mathbb{R}_{\geq 0}$. The value $r(e)$ is called the resistance of the edge $e$, and we will assume throughout this paper that $r(e)=1$ for each edge $e$. After fixing a source vertex $s$ and sink vertex $t$ there exists a unique potential function $\pi: V(G)\to [0,1]$ with the property that
$\pi(s)=0$, $\pi(t)=1$ and at each vertex $u\neq s,t$ we have
\begin{equation} \label{current-flow}
\sum_{w\in N_G(u)}(\pi(w)-\pi(u))=0.
\end{equation}
The physical interpretation of the potential $\pi$ is that in an edge $(u,v)$ there is a current $i_{vu}=\pi(v)-\pi(u)$ that flows with the convention that if it is negative, then we think of it as a current $\pi(u)-\pi(v)$ flowing from $u$ to $v$. Equation~\eqref{current-flow} simply says that for a vertex $u\neq s,t$ the same amount of current flows into $u$ as it leaves. The effective resistance between $s$ and $t$ is then
$$R_{\eff}(s,t)=\left(\sum_{v\in N_G(s)}(\pi(v)-\pi(s))\right)^{-1}=\left(\sum_{v\in N_G(s)}\pi(v)\right)^{-1}.$$
For a physical interpretation, we again refer to the book \cite{bollobas1998modern}. Lemma~\ref{effective-combinatorial} connects the effective resistance with graph theory. Effective resistances are also related to random walks.
\bigskip

Recall the RSW-formula (Theorem~\ref{thm: RSW-formula}):
$$\frac{F_2(G)}{T(G)}=\frac{1}{4}\sum_{e\in E}R_{\eff}(e^-,e^+)^2+\frac{3}{4}\sum_{e\in E}(R_{\eff}(q,e^-)-R_{\eff}(q,e^+))^2.$$
Let
$$\mathrm{RSW}_1=\frac{1}{4}\sum_{e\in E}R_{\eff}(e^-,e^+)^2\ \ \ \text{and}\ \ \ \mathrm{RSW}_2=\frac{3}{4}\sum_{e\in E}(R_{\eff}(q,e^-)-R_{\eff}(q,e^+))^2.$$
We collect a few observations regarding $\mathrm{RSW}_1$ and $\mathrm{RSW}_2$.

\begin{Lemma} \label{effective-combinatorial}
For $u,v\in V(G)$ such that $u\neq v$, let $G_{u,v}$ be the graph obtained by identifying the vertices $u$ and $v$ in $G$. Then
$$R_{\eff}(u,v)=\frac{T(G_{u,v})}{T(G)}.$$
In particular, if $(u,v)\in E(G)$, then $R_{\eff}(u,v)$ is the probability that a uniformly chosen random spanning tree contains the edge $(u,v)$.
\end{Lemma}

\begin{Lemma}[Rayleigh's monotonicity]
Adding edges to a graph cannot increase $R_{\eff}(u,v)$.
\end{Lemma}

\begin{Lemma} \label{sum-resistances} If $G$ is a connected graph on $n$ vertices, then
$$\sum_{e\in E}R_{\eff}(e^-,e^+)=n-1.$$
\end{Lemma}

\begin{proof} Let $\textbf{T}$ be a uniformly chosen random spanning tree. Then we have
$$\sum_{e\in E}R_{\eff}(e^-,e^+)=\sum_{e\in E}\mathbb{P}(e\in \textbf{T})=\mathbb{E}[|\textbf{T}|]=n-1.$$
\end{proof}

\begin{Lemma} \label{effective-resistance-degree-d}
Let $G$ be a graph on $n$ vertices and suppose that the vertex $v$ has degree $d$. If $u$ is a neighbor of $v$, then
$$R_{\eff}(v,u)\geqslant \frac{d+n-1}{dn}.$$
\end{Lemma}

\begin{proof} First, let us construct a graph $G'$ by adding each edge to $G$ that is non-incident to $v$. Then the resulting graph $G'$ is a $K_{n-1}$ with the extra vertex $v$ connected to $d$ elements of $K_{n-1}$. Since adding edges to the graph can only decrease the effective resistance, it is enough to compute $R^{G'}_{\eff}(u,v)$ in the resulting graph $G'$. Next we show that this effective resistance is exactly $\frac{d+n-1}{dn}$.
Let $w_1$ be a vertex of degree $n-1$ and $w_2$ be a vertex of degree $n-2$ different from $u$ and $v$. If $2\leqslant d\leqslant n-2$ these vertices exist. If $d=1$ or $d=n-1$ the claim is still true: if $d=1$, the incident edge is a bridge and $R_{\eff}(u,v)=1$;
if $d=n-1$, then the graph $G'$ is $K_n$ and $R_{\mathrm{eff}}(u,v)=\frac{2}{n}$. For $2\leqslant d\leqslant n-2$, we can use vertex symmetry; the potential equations are the following:
$$\pi(v)=0,\ \ \ \pi(u)=1,$$
$$(n-1)\pi(w_1)=\pi(v)+\pi(u)+(d-2)\pi(w_1)+(n-d-1)\pi(w_2)=1+(d-2)\pi(w_1)+(n-d-1)\pi(w_2)$$
$$(n-2)\pi(w_2)=\pi(u)+(d-1)\pi(w_1)+(n-d-2)\pi(w_2)=1+(d-1)\pi(w_1)+(n-d-2)\pi(w_2).$$
The unique solution of this system of equations is
$$\pi(w_1)=\frac{n-1}{d+n-1},\  \ \ \pi(w_2)=\frac{n}{d+n-1}$$
and
$$R^{G'}_{\eff}(u,v)=\left(\pi(u)+(d-1)\pi(w_1)\right)^{-1}=\frac{d+n-1}{dn}.$$
Hence $R^G_{\eff}(u,v)\geqslant \frac{d+n-1}{dn}$.

\end{proof}

\begin{Lemma}[Richman, Shokrieh and Wu \cite{richman2023counting}] \label{RSW-1-lower-bound}
Let $G$ be a connected graph with $n$ vertices and $m$ edges. Then
$$\mathrm{RSW}_1\geqslant \frac{(n-1)^2}{4m}.$$
\end{Lemma}

\begin{proof} By the Cauchy--Schwarz inequality we have
$$\mathrm{RSW}_1=\frac{1}{4}\sum_{e\in E}R_{\eff}(e^-,e^+)^2\geqslant \frac{1}{4m}\left(\sum_{e\in E}R_{\eff}(e^-,e^+)\right)^2=\frac{(n-1)^2}{4m}.$$
\end{proof}

\begin{Lemma} \label{RSW-1-RSW-2}
Let $G$ be a connected graph on $n$ vertices. Then
$$\mathrm{RSW}_2\geqslant \frac{6}{n}\mathrm{RSW}_1.$$
If $G=K_n$, then this inequality holds with equality. 
\end{Lemma}

\begin{proof} Let
$$\mathrm{RSW}_2(q)=\frac{3}{4}\sum_{e\in E}(R_{\eff}(q,e^-)-R_{\eff}(q,e^+))^2\geqslant \frac{3}{4}\sum_{e\in E \atop q\in e}R_{\eff}(e^-,e^+)^2.$$
The expression $\mathrm{RSW}_2(q)$ is independent of $q$.
Note that if $G$ is the complete graph $K_n$, then $R_{\eff}(q,u)=R_{\eff}(q,v)$ for $u,v\neq q$ by symmetry, so we have equality in this bound.

Now, let us average it for all $q\in V$. Then
$$\mathrm{RSW}_2=\frac{1}{n}\sum_{q\in V}\mathrm{RSW}_2(q)\geqslant \frac{1}{n}\sum_{q\in V}\frac{3}{4}\sum_{e\in E \atop q\in e}R_{\eff}(e^-,e^+)^2=\frac{6}{n}\cdot \frac{1}{4}\sum_{e\in E}R_{\eff}(e^-,e^+)^2=\frac{6}{n}\mathrm{RSW}_1.$$
\end{proof}

\begin{Lemma} \label{RSW-2-lower-bound}
Let $G$ be a connected graph on $n$ vertices with a vertex of degree $d$. Then
$$\mathrm{RSW}_2\geqslant \frac{3}{4d}\left(1+\frac{d-1}{n}\right)^2.$$
\end{Lemma}

\begin{proof}
Let us choose $q=v$ in the definition of $\mathrm{RSW}_2$. Then for each edge incident to $v$ we have
$$R_{\eff}(u,v)-R_{\eff}(v,v)=R_{\eff}(u,v)\geqslant \frac{d+n-1}{dn}$$
by Lemma~\ref{effective-resistance-degree-d}. The claim immediately follows.
\end{proof}

\begin{Lemma} \label{2-forests-complete-graph} Let $G$ be a simple connected graph on $n\geq 2$ vertices. Then
$$\frac{F_2(G)}{T(G)}\geqslant \frac{F_2(K_n)}{T(K_n)}=\left(1+\frac{6}{n}\right)\frac{n-1}{2n}>\frac{1}{2}.$$
\end{Lemma}

\begin{proof} By combining Lemmas~\ref{RSW-1-lower-bound} and \ref{RSW-1-RSW-2} we have 
\begin{align*}
\frac{F_2(G)}{T(G)}&=\mathrm{RSW}_1+\mathrm{RSW}_2  && (\text{Theorem~\ref{thm: RSW-formula}})\\
&\geqslant \left(1+\frac{6}{n}\right)\mathrm{RSW}_1  && (\text{Lemma~\ref{RSW-1-RSW-2}}) \\
&\geqslant \left(1+\frac{6}{n}\right)\frac{(n-1)^2}{4m} && (\text{Lemma~\ref{RSW-1-lower-bound}})\\
&\geqslant \left(1+\frac{6}{n}\right)\frac{(n-1)^2}{4\binom{n}{2}} && \left(\binom{n}{2}\geqslant m\right)\\
&=\left(1+\frac{6}{n}\right)\frac{n-1}{2n}.
\end{align*}
Note that for the complete graph $K_n$ we have equality in each step, so
$$\frac{F_2(K_n)}{T(K_n)}=\left(1+\frac{6}{n}\right)\frac{n-1}{2n}.$$
\end{proof}

\begin{Lemma} \label{cut-resistance}
Let $G$ be a connected graph and let $C\subseteq E$ be a cut. Then
$$\sum_{e\in C}R_{\eff}(e^-,e^+)\geqslant 1.$$
\end{Lemma}

\begin{proof}
Let $\textbf{T}$ be a uniformly chosen random spanning tree. Then
$$\sum_{e\in C}R_{\eff}(e^-,e^+)=\sum_{e\in C}\mathbb{P}(e\in \textbf{T})=\mathbb{E}[|C\cap \textbf{T}|]\geqslant 1$$
since every spanning tree contains at least one edge from a cut.
\end{proof}

\begin{Rem} By Rayleigh's monotonicity, it is also true that for each edge $e\in C$ we have $R_{\eff}(e^-,e^+)\geqslant \frac{1}{|C|}.$
\end{Rem}

\begin{Lemma} \label{edge-number-estimate}
Let $G$ be a simple graph on $n$ vertices with smallest cut of size $\delta$ and smallest degree $\delta_v$. Suppose that $\delta<\delta_v$. Then $\delta\leqslant \frac{n}{2}-2$, and $G$ has at most  $\binom{n}{2}-(\delta+1)(n-\delta-1)+\delta$ edges.
\end{Lemma}

\begin{proof}
Let $E(S,V\setminus S)$ be a cut of size $\delta$ where $|S|=k\leqslant \frac{n}{2}$. Since $\delta<\delta_v$ we have $k\geqslant 2$. Since $|E(S,V\setminus S)|=\delta$, there must be a vertex $u\in S$ such that 
$$d(u)\leqslant |S|-1+\frac{\delta}{|S|}=k-1+\frac{\delta}{k}.$$
Note that $d(u)\geqslant \delta_v>\delta$, thus
$$\delta<k-1+\frac{\delta}{k}.$$
This is equivalent to $\delta<k$, that is, $\delta+1\leqslant k$.
 If $k=\delta+1$, then there would be a vertex $u\in S$ without neighbor in $V\setminus S$ and its degree would be at most $|S|-1=\delta<\delta_v$ contradicting the definition of $\delta_v$. So $\frac{n}{2}\geqslant k\geqslant \delta+2$ proving the first statement. Using only $\frac{n}{2}\geqslant k\geqslant \delta+1$ we also have
$$e(G)\leqslant \binom{n}{2}-k(n-k)+\delta\leqslant \binom{n}{2}-(\delta+1)(n-\delta-1)+\delta.$$
\end{proof}

We will also need some explicit upper bound for $Q_n$.

\begin{Lemma} \label{Q_n-upper-bound}
For every $n\geqslant 1$ we have
$$Q_n\leqslant \exp\left(\frac{1}{2}+\frac{5}{2n}-\frac{3}{2n^2}\right).$$
\end{Lemma}

We postpone the proof of Lemma~\ref{Q_n-upper-bound} because it fits more naturally into Section~\ref{incusion-exclusion-sect}.

\begin{Rem}
Much more precise statements are known, but we need some inequality valid for all $n$. For example,
$$Q_n=\sqrt{e}\left(1+\frac{5}{2n}+\frac{11}{8n^2}-\frac{203}{16n^3}-\frac{17207}{384n^4}-\frac{3607}{768n^5}+\frac{1408301}{3072n^6}+\frac{8181503}{6144n^7}+\dots\right).$$
\end{Rem}

Finally, we will use the following basic inequalities.

\begin{Lemma} \label{Lemma-ln(1+x)}
(a) For $x\in (0,1)$ we have
\begin{equation} \label{ln(1+x)}
x-\frac{x^2}{2}<x-\frac{x^2}{2}+\frac{x^3}{3}-\frac{x^4}{4}<\ln(1+x)<x-\frac{x^2}{2}+\frac{x^3}{3}<x.
\end{equation}
(b) For $x\in [0,1]$ we have
\begin{equation} \label{ln(1+x)-2}
\frac{2x}{2+x}\leqslant \ln(1+x)\leqslant\frac{x(2+x)}{2(1+x)}.
\end{equation}
(c) For $\delta\geqslant 1$ we have
\begin{equation} \label{func-ineq-0}
 \delta \ln\left(1+\frac{1}{\delta}\right)>1-\frac{1}{2\delta}
\end{equation}
and
\begin{equation} \label{func-ineq}
\left(1+\frac{1}{2\delta}\right)\cdot \delta \ln\left(1+\frac{1}{\delta}\right)>1.
\end{equation}
\end{Lemma}

\begin{proof}
Part (a) immediately follows from $\ln(1+x)=\sum_{k=1}^{\infty}(-1)^{k+1}\frac{x^k}{k}$ and the fact that for $x\in (0,1)$ we have $\frac{x^k}{k}>\frac{x^{k+1}}{k+1}$.

Part (b) follows from the fact that the derivatives satisfy
$$\left(\ln(1+x)-\frac{2x}{2+x}\right)'=\frac{x^2}{(1+x)(2+x)^2}>0$$
and
$$\left(\frac{x(2+x)}{2(1+x)}-\ln(1+x)\right)'=\frac{x^2}{2(1+x)^2}>0$$
and both differences take the value $0$ at $x=0$.

The first inequality in part (c) follows from part (a) for $\delta >1$ by taking $x=\frac{1}{\delta}$ and by continuity to $\delta=1$.
To prove the second inequality in part (c) observe that with $x=\frac{1}{\delta}$ we obtain the first inequality in part (b). (Also, observe that the second inequality is actually stronger:
$\frac{2\delta}{2\delta+1}>1-\frac{1}{2\delta}$.)

\end{proof}

\medskip

Now we are ready to prove Theorem~\ref{forest-tree-ratio}.

\begin{proof}[Proof of Theorem~\ref{forest-tree-ratio}]
First we prove the statement for $2\leqslant n\leqslant 5$. We have
$$\frac{F(G)}{T(G)}\geqslant 1+\frac{F_2(G)}{T(G)}\geqslant 1+\frac{(n+6)(n-1)}{2n^2}$$
by Lemma~\ref{2-forests-complete-graph}. For $2\leqslant n\leqslant 5$ the right-hand side is at least
$1+\frac{11\cdot 4}{50}=\frac{47}{25}>\sqrt{3}>\sqrt{e}$.
Henceforth, assume that $n\geqslant 6$. 

Let $\delta$ be the size of the smallest cut and let $C$ be a cut of size $\delta$. Furthermore, let
$R_C=\sum_{e\in C}R_{\eff}(e^-,e^+).$
From Lemma~\ref{cut-resistance} we know that $R_C\geqslant 1$.
\medskip

\noindent \textbf{Case 1:} $m\leqslant \frac{(n-1)^2}{3}$. Then
\begin{align*}
\frac{F(G)}{T(G)}&\geqslant 1+\frac{F_2(G)}{T(G)} \\
&=1+\mathrm{RSW}_1+\mathrm{RSW}_2 && (\text{Theorem~\ref{thm: RSW-formula}})\\
&\geqslant 1 +\mathrm{RSW}_1\\
&\geqslant 1 + \frac{(n-1)^2}{4m}  && (\text{Lemma~\ref{RSW-1-lower-bound}})\\
&\geqslant 1 +\frac{3}{4}  && \left(m\leqslant \frac{(n-1)^2}{3}\right)\\
&>\sqrt{e}.
\end{align*}
\medskip

\noindent \textbf{Case 2:} $\delta>\frac{n}{4}$. Then
\begin{align*}
\ln\left(\frac{F(G)}{T(G)}\right)&\geqslant  \frac{F_2(G)}{T(G)}\cdot \delta \ln\left(1+\frac{1}{\delta}\right) && (\text{Theorem~\ref{ratio-minimum}})\\
&\geqslant \frac{F_2(K_n)}{T(K_n)} \cdot \delta \ln\left(1+\frac{1}{\delta}\right) && (\text{Lemma~\ref{2-forests-complete-graph}})\\
&=\left(1+\frac{6}{n}\right)\frac{n-1}{2n}\cdot \delta \ln\left(1+\frac{1}{\delta}\right) && (\text{Lemma~\ref{2-forests-complete-graph}})\\
&\geqslant \left(1+\frac{6}{n}\right)\frac{n-1}{2n}\cdot \left(1-\frac{1}{2\delta}\right) && (\text{Ineq.
} \eqref{func-ineq-0})\\
&\geqslant \left(1+\frac{6}{n}\right)\frac{n-1}{2n}\cdot \left(1-\frac{2}{n}\right) && \left(\delta>\frac{n}{4}\right)\\
&=\frac{1}{2}+\frac{3}{2n}-\frac{8}{n^2}+\frac{6}{n^3}\\
&>\frac{1}{2}.
\end{align*}
In the last step we used that $\frac{3}{2n}>\frac{8}{n^2}$ if $n>\frac{16}{3}$ which is true since $n\geqslant 6$.
\medskip

\noindent \textbf{Case 3:} $\delta\leqslant \frac{n}{4}$ and $m> \frac{(n-1)^2}{3}$. In the following calculation we will use two minor observations. First, if $C=E(S,V\setminus S)$ with $|S|=k$, then we have
\begin{equation} \label{m-delta-upper}
m-\delta\leqslant \binom{k}{2}+\binom{n-k}{2}=\binom{n}{2}-k(n-k)\leqslant \binom{n}{2}-(n-1)=\binom{n-1}{2}.
\end{equation}
Furthermore, for the polynomial 
$$P(x)=\frac{1}{\delta}x^2+\frac{1}{m-\delta}(n-1-x)^2$$
we have $\min_{x\geqslant 1}P(x)=P(1)$ if $\frac{m}{n-1}\geqslant \delta$ since this quadratic polynomial has positive leading coefficient and a minimum at
$$\frac{\frac{1}{m-\delta}(n-1)}{\frac{1}{\delta}+\frac{1}{m-\delta}}=\frac{\delta (n-1)}{m}\leqslant 1.$$
In our case the condition $\frac{m}{n-1}\geqslant \delta$ is satisfied as $\frac{m}{n-1}> \frac{n-1}{3}\geqslant \frac{n}{4}\geqslant \delta$. 
\medskip

With these observations in mind we have
\begin{align*}
\ln\left(\frac{F(G)}{T(G)}\right)&\geqslant  \frac{F_2(G)}{T(G)}\cdot \delta \ln\left(1+\frac{1}{\delta}\right) && (\text{Theorem~\ref{ratio-minimum}})\\
&=(\mathrm{RSW}_1+\mathrm{RSW}_2)\cdot \delta \ln\left(1+\frac{1}{\delta}\right) && (\text{Theorem~\ref{thm: RSW-formula}})\\
&\geqslant \left(1+\frac{6}{n}\right)\mathrm{RSW}_1\cdot \delta \ln\left(1+\frac{1}{\delta}\right) && (\text{Lemma~\ref{RSW-1-RSW-2}})\\
&=\left(1+\frac{6}{n}\right)\cdot \frac{1}{4}\sum_{e\in E}R_{\eff}(e^-,e^+)^2\cdot \delta \ln\left(1+\frac{1}{\delta}\right)  && (\text{Def. of}\ \mathrm{RSW_1})\\
&=\left(1+\frac{6}{n}\right)\cdot \frac{1}{4}\left(\sum_{e\in C}R_{\eff}(e^-,e^+)^2+\sum_{e\in E\setminus C}R_{\eff}(e^-,e^+)^2\right)\cdot \delta \ln\left(1+\frac{1}{\delta}\right)\\
&\geqslant \frac{1}{4}\left(1+\frac{6}{n}\right)\cdot \delta \ln\left(1+\frac{1}{\delta}\right) \\
&\quad \cdot \left(\frac{1}{\delta}\left(\sum_{e\in C}R_{\eff}(e^-,e^+)\right)^2+\frac{1}{m-\delta}\left(\sum_{e\in E\setminus C}R_{\eff}(e^-,e^+)\right)^2\right) && (\text{Cauchy-Schwarz})\\
&=\left(1+\frac{6}{n}\right)\cdot \frac{1}{4}\left(\frac{1}{\delta}R_C^2+\frac{1}{m-\delta}(n-1-R_C)^2\right)\cdot \delta \ln\left(1+\frac{1}{\delta}\right)\\
&\geqslant \left(1+\frac{6}{n}\right)\cdot \frac{1}{4}\left(\frac{1}{\delta}+\frac{1}{m-\delta}(n-2)^2\right)\cdot \delta \ln\left(1+\frac{1}{\delta}\right) && (P(R_C)\geqslant P(1))\\
&\geqslant \left(1+\frac{6}{n}\right)\cdot \frac{1}{4}\left(\frac{1}{\delta}+\frac{1}{\binom{n-1}{2}}(n-2)^2\right)\cdot \delta \ln\left(1+\frac{1}{\delta}\right) && (\text{Ineq.~\eqref{m-delta-upper}})\\
&= \left(1+\frac{6}{n}\right)\cdot \frac{1}{4}\left(\frac{1}{\delta}+2-\frac{2}{n-1}\right)\cdot \delta \ln\left(1+\frac{1}{\delta}\right)\\
&=\left(\frac{1}{2}+\frac{1}{4\delta}+\frac{3}{n}-\frac{1}{2(n-1)}-\frac{3}{n(n-1)}+\frac{3}{2\delta n}\right)\cdot \delta \ln\left(1+\frac{1}{\delta}\right)\\
&=\left(\frac{1}{2}+\frac{1}{4\delta}+\frac{5n-12}{2n(n-1)}+\frac{3}{2\delta n}\right) \cdot \delta \ln\left(1+\frac{1}{\delta}\right)\\
&\geqslant\frac{1}{2}\left(1+\frac{1}{2\delta}\right)\cdot \delta \ln\left(1+\frac{1}{\delta}\right) && (n\geqslant 3)\\
&>\frac{1}{2}. && (\text{Ineq~\eqref{func-ineq}})
\end{align*}
Hence $Q(G)>\sqrt{e}$ in all cases.
\end{proof}

Next we prove Theorems~\ref{forest-tree-ratio-2} and \ref{forest-tree-ratio-2-quantitative}. The proof will be very similar to the previous one.

\begin{proof}[Proof of Theorems~\ref{forest-tree-ratio-2} and \ref{forest-tree-ratio-2-quantitative}]
We will assume that $n\geqslant 6$. For $n\leqslant 5$ it is easy to check the statement either by computer or very slightly refining the estimates given below. Let $G$ be a graph with $|E(G)|=\binom{n}{2}-h$ edges, where $h=cn$. 

Let $\delta$ be the size of the smallest cut and let $C$ be a cut of size $\delta$. Furthermore, let
$R_C=\sum_{e\in C}R_{\eff}(e^-,e^+).$
From Lemma~\ref{cut-resistance} we know that $R_C\geqslant 1$.
\medskip

\noindent \textbf{Case 1:} $m\leqslant \frac{(n-1)^2}{3}$. Then
\begin{align*}
\frac{F(G)}{T(G)}&\geqslant 1+\frac{F_2(G)}{T(G)}\\
&= 1 +\mathrm{RSW}_1+\mathrm{RSW}_2 && (\text{Theorem~\ref{thm: RSW-formula}})\\
&\geqslant 1 +\left(1+\frac{6}{n}\right)\mathrm{RSW}_1 && (\text{Lemma~\ref{RSW-1-RSW-2}})\\
&\geqslant 1 + \left(1+\frac{6}{n}\right)\frac{(n-1)^2}{4m}  && (\text{Lemma~\ref{RSW-1-lower-bound}})\\
&\geqslant 1 +\left(1+\frac{6}{n}\right)\frac{3}{4} && \left(m\leqslant \frac{(n-1)^2}{3}\right)\\
&>\exp\left(\frac{1}{2}+\frac{5}{2n}\right) && (n\geqslant 7).
\end{align*}
The last inequality holds true if $n\geqslant 7$, for $1\leqslant n\leqslant 6$ one can check by a computer that $1+\frac{3}{4}\left(1+\frac{6}{n}\right)>Q_n$. So in this case $Q(G)>Q_n$ for all $n$.
\medskip

\noindent \textbf{Case 2:} $\delta_v>\delta\geqslant \frac{n}{4}$. By Lemma~\ref{edge-number-estimate} we have $\delta\leqslant \frac{n}{2}-2$ and the number of edges is at most
$$m\leqslant \binom{n}{2}-(\delta+1)(n-\delta-1)+\delta\leqslant \binom{n}{2}-\left(\frac{n}{4}+1\right)\left(n-\frac{n}{4}-1\right)+\frac{n}{4}=\frac{5}{16}n^2-\frac{3n}{4}+1<\frac{(n-1)^2}{3},$$
where the last inequality holds true if $n>4$, and the first inequality is true, because the function
$\psi(x)=(x+1)(n-x-1)-x$ has derivative $\psi'(x)=n-3-2x\geqslant 1$ on $[n/4,n/2-2]$. So we are actually in Case 1. For $n\leqslant 4$, the condition $\delta_v>\delta$ cannot be satisfied.
\medskip

\noindent \textbf{Case 3:} $\delta_v>\delta$, $\delta<\frac{n}{4}$ and $m> \frac{(n-1)^2}{3}$. In the forthcoming computation we will use two minor observations. First, if $C=E(S,V\setminus S)$ with $|S|=k$, then we have
\begin{equation} \label{m-delta-upper-2}
m-\delta\leqslant \binom{k}{2}+\binom{n-k}{2}=\binom{n}{2}-k(n-k)\leqslant \binom{n}{2}-2(n-2)<\frac{(n-2)^2}{2}
\end{equation}
since $\delta<\delta_v$ means that $2\leqslant k\leqslant n-2$.
Furthermore, for the polynomial 
$$P(x)=\frac{1}{\delta}x^2+\frac{1}{m-\delta}(n-1-x)^2$$
we have $\min_{x\geqslant 1}P(x)=P(1)$ if $\frac{m}{n-1}\geqslant \delta$ since this quadratic polynomial has positive leading coefficient and a minimum at
$$\frac{\frac{1}{m-\delta}(n-1)}{\frac{1}{\delta}+\frac{1}{m-\delta}}=\frac{\delta (n-1)}{m}\leqslant 1.$$
In our case the condition $\frac{m}{n-1}\geqslant \delta$ is satisfied as $\frac{m}{n-1}> \frac{n-1}{3}\geqslant \frac{n}{4}>\delta$. 
\medskip

With these observations in mind we have
\begin{align*}
\ln\left(\frac{F(G)}{T(G)}\right)&\geqslant  \frac{F_2(G)}{T(G)}\cdot \delta \ln\left(1+\frac{1}{\delta}\right) && (\text{Theorem~\ref{ratio-minimum}})\\
&=(\mathrm{RSW}_1+\mathrm{RSW}_2)\cdot \delta \ln\left(1+\frac{1}{\delta}\right) && (\text{Theorem~\ref{thm: RSW-formula}})\\
&\geqslant \left(1+\frac{6}{n}\right)\mathrm{RSW}_1\cdot \delta \ln\left(1+\frac{1}{\delta}\right) && (\text{Lemma~\ref{RSW-1-RSW-2}})\\
&=\left(1+\frac{6}{n}\right)\cdot \frac{1}{4}\sum_{e\in E}R_{\eff}(e^-,e^+)^2\cdot \delta \ln\left(1+\frac{1}{\delta}\right) && (\text{Def of.}\ \mathrm{RSW_1})\\
&=\left(1+\frac{6}{n}\right)\cdot \frac{1}{4}\left(\sum_{e\in C}R_{\eff}(e^-,e^+)^2+\sum_{e\in E\setminus C}R_{\eff}(e^-,e^+)^2\right)\cdot \delta \ln\left(1+\frac{1}{\delta}\right)\\
&\geqslant \frac{1}{4}\left(1+\frac{6}{n}\right)\cdot \delta \ln\left(1+\frac{1}{\delta}\right)\\
&\quad \cdot \left(\frac{1}{\delta}\left(\sum_{e\in C}R_{\eff}(e^-,e^+)\right)^2+\frac{1}{m-\delta}\left(\sum_{e\in E\setminus C}R_{\eff}(e^-,e^+)\right)^2\right)  && (\text{Cauchy-Schwarz})\\
&=\left(1+\frac{6}{n}\right)\cdot \frac{1}{4}\left(\frac{1}{\delta}R_C^2+\frac{1}{m-\delta}(n-1-R_C)^2\right)\cdot \delta \ln\left(1+\frac{1}{\delta}\right)\\
&\geqslant \left(1+\frac{6}{n}\right)\cdot \frac{1}{4}\left(\frac{1}{\delta}+\frac{1}{m-\delta}(n-2)^2\right)\cdot \delta \ln\left(1+\frac{1}{\delta}\right) && (P(R_C)\geqslant P(1))\\
&\geqslant \left(1+\frac{6}{n}\right)\cdot \frac{1}{4}\left(\frac{1}{\delta}+2\right)\cdot \delta \ln\left(1+\frac{1}{\delta}\right) && (\text{Ineq.~\eqref{m-delta-upper-2}})\\
&>\left(1+\frac{6}{n}\right)\frac{1}{2} && (\text{Ineq.~\eqref{func-ineq}})\\
&>\ln(Q_n). && (\text{Lemma~\ref{Q_n-upper-bound}})
\end{align*}
Hence $Q(G)>Q_n$.
\bigskip

\noindent \textbf{Case 4:} $\delta_v=\delta$. Then
\begin{align*}
\ln\left(\frac{F(G)}{T(G)}\right)&\geqslant  \frac{F_2(G)}{T(G)}\cdot \delta \ln\left(1+\frac{1}{\delta}\right) && (\text{Theorem~\ref{ratio-minimum}})\\
&=(\mathrm{RSW}_1+\mathrm{RSW}_2)\cdot \delta \ln\left(1+\frac{1}{\delta}\right) && (\text{Theorem~\ref{thm: RSW-formula}})\\
&\geqslant \left(\frac{(n-1)^2}{4m}+\frac{3}{4\delta}\left(1+\frac{\delta-1}{n}\right)^2\right)\cdot \delta \ln\left(1+\frac{1}{\delta}\right) && (\text{Lemma~\ref{RSW-1-lower-bound} and \ref{RSW-2-lower-bound}})
\end{align*}
Let
\begin{equation} \label{def:fmnd}
f(n,m,\delta):=\left(\frac{(n-1)^2}{4m}+\frac{3}{4\delta}\left(1+\frac{\delta-1}{n}\right)^2\right)\cdot \delta \ln\left(1+\frac{1}{\delta}\right).
\end{equation}
For each $n$ and $\delta$ let 
$$m(n,\delta)=\max \left\{ q\in \left\{1,\dots ,\binom{n}{2}\right\}\ \bigg|\ f(n,q,\delta)>\ln(Q_n)\right\},$$
and let
$$m_n=\min_{1\leqslant \delta \leqslant n-1}m(n,\delta).$$
Furthermore, let 
$$h_n:=\binom{n}{2}-m_n.$$
For fixed $n,\delta$, the function $m\mapsto f(n,m,\delta)$ is decreasing. Hence $m\leqslant m_n\leqslant m(n,\delta)$ implies $f(n,m,\delta)>\ln(Q_n)$.
Thus, if $m\leqslant m_n$, equivalently $h\geqslant h_n$, then $Q(G)>Q_n$. 
With a computer, we can easily compute $m_n$ and $h_n$. In the following, we give their values for $n\leqslant 20$.
\bigskip

\begin{table}[htbp]
\centering
\caption{Data for $n\leqslant 20$.}
\label{tab: table 1}

\begin{tabular}{|c|cccccccccccccccccc|} \hline
$n$ &   3 & 4 & 5 & 6 & 7 & 8 & 9 & 10 & 11 & 12 & 13 & 14 & 15 & 16 & 17 & 18 & 19  & 20 \\ \hline
$m_n$ &  2& 4& 7& 12& 17& 24& 32& 40& 50& 60& 72& 85& 98& 113& 129& 145& 163& 182 \\ \hline
$h_n$ &  1& 2& 3& 3& 4& 4& 4& 5& 5& 6& 6& 6& 7& 7& 7& 8& 8& 8\\ \hline
$\lfloor 0.49n+1/2\rfloor$ & 1 & 2 & 2 & 3 & 3 & 4 & 4 & 5 & 5 & 6 & 6 & 7 & 7 & 8 & 8 & 9 & 9 & 10 \\ \hline
\end{tabular}
\end{table}

\bigskip

We can see that Theorem~\ref{forest-tree-ratio-2-quantitative} holds true for $8\leqslant n\leqslant 20$.

\begin{align*}
f(n,m,\delta)&\geqslant \left(\frac{(n-1)^2}{4\left(\binom{n}{2}-cn\right)}+\frac{3}{4\delta}\left(1+\frac{\delta-1}{n}\right)^2\right)\cdot \delta \ln\left(1+\frac{1}{\delta}\right)\\
&\geqslant \left(\frac{(n-1)^2}{4\left(\binom{n}{2}-c(n-1)\right)}+\frac{3}{4\delta}\left(1+\frac{\delta-1}{n}\right)^2\right)\cdot \delta \ln\left(1+\frac{1}{\delta}\right)\\
&= \left(\frac{n-1}{2(n-2c)}+\frac{3}{4\delta}\left(1+\frac{\delta-1}{n}\right)^2\right)\cdot \delta \ln\left(1+\frac{1}{\delta}\right)\\
&\geqslant  \frac{1}{2}\left(\frac{n-1}{n-2c}+\frac{3}{2\delta}\left(1+\frac{\delta-1}{n}\right)^2\right)\cdot \left(1-\frac{1}{2\delta}\right)\\
&=\frac{1}{2}\left(1+\frac{2c-1}{n-2c}+\frac{3}{2\delta}\left(1+\frac{\delta-1}{n}\right)^2\right)\cdot \left(1-\frac{1}{2\delta}\right).
\end{align*}
Let
$$L_n(c,\delta):=\frac{1}{2}\left(1+\frac{2c-1}{n-2c}+\frac{3}{2\delta}\left(1+\frac{\delta-1}{n}\right)^2\right)\cdot \left(1-\frac{1}{2\delta}\right).$$

It is enough to prove that
\begin{equation} \label{Lncd-inequality}
L_n(c,\delta)\geqslant \frac{1}{2}\left(1+\frac{5}{n
}-\frac{3}{n^2}\right)
\end{equation}
as the right-hand side is larger than $\ln(Q_n)$ by Lemma~\ref{Q_n-upper-bound}. One can show that for $c>\frac{3-\sqrt{6}}{2}$ inequality \eqref{Lncd-inequality} holds true for sufficiently large $n$. Nevertheless, we follow another route. First, we show that for $c\geqslant 1/2$ the inequality \eqref{Lncd-inequality} holds true for $n\geqslant 20$. Then with a slight modification of the above method we show that for $c< 1/2$, $n\geqslant 20$ and $h\geqslant \frac{n}{4}+\frac{9}{2}$ we also have $Q(G)>Q_n$. This will resolve both Theorems~\ref{forest-tree-ratio-2} and \ref{forest-tree-ratio-2-quantitative}.

First,  consider the case $c\geqslant \frac{1}{2}$. For fixed $n,\delta$, the function $L_n(c,\delta)$ is increasing in $c$. Hence, when $c\geqslant \frac{1}{2}$, it is enough to consider $c=\frac{1}{2}$.
Thus it is enough to show that for $n\geqslant 20$ we have $L_n(1/2,\delta)>\frac{1}{2}\left(1+\frac{5}{n}-\frac{3}{n^2}\right)$. We have
\begin{equation*}L_n(1/2,\delta)-\frac{1}{2}\left(1+\frac{5}{n}-\frac{3}{n^2}\right)=\frac{(6\delta^3-8\delta^2 n+4\delta n^2)-3(\delta^2+6\delta n+n^2)+(6n + 12\delta)- 3}{8\delta^2n^2}
\end{equation*}
Since $6n+12\delta>3$ it is enough to prove that
$$6\delta^3-8\delta^2 n+4\delta n^2>3(\delta^2+6\delta n+n^2)$$
if $n\geqslant 20$. Let $\alpha=\frac{\delta}{n}$. Then we need to prove that
$$n>g(\alpha):=\frac{3(\alpha^2+6\alpha+1)}{6\alpha^3-8\alpha^2+4\alpha}.$$
If $\delta=1$, let $e$ be a bridge. Every spanning tree contains $e$, and the map $T\mapsto T-e$ injects the spanning trees into the two-component forests. Hence
$F_2(G)\geqslant T(G)$ and $Q(G)\geqslant 2$. Note that for $n\geqslant 20$ we have
$$\ln(Q_n)<\frac{1}{2}+\frac{5}{2\cdot 20}-\frac{3}{2\cdot 20^2}<\frac{5}{8}<\frac{2}{3}<\ln(2)$$
by Lemma~\ref{Q_n-upper-bound} and by \eqref{ln(1+x)-2}.
Thus, $Q_n<2$ for $n\geq 20$ and we are done in this case as well. So we can assume that $\delta\geqslant 2$ and consequently $\alpha\geqslant \frac{2}{n}$. 
We have
$$g'(\alpha)=-\frac{3(3\alpha^4 + 36\alpha^3 - 17\alpha^2 - 8\alpha + 2)}{2(3\alpha^2 - 4\alpha + 2)^2\alpha^2}.$$
Let
$$P(\alpha)=3\alpha^4 + 36\alpha^3 - 17\alpha^2 - 8\alpha + 2.$$
Then for $0\leqslant \alpha \leqslant \frac{1}{10}$ we have
$$P(\alpha)\geqslant 2-\frac{17}{100}-\frac{8}{10}>0$$
so $g$ is decreasing on this interval. 
For $\alpha\geqslant \frac{1}{10}$, the inequality $g(\alpha)<20$ is equivalent to
$$ 120\alpha^3-163\alpha^2+62\alpha-3>0. $$
This cubic has discriminant $-2247488<0$, so it has only one real root; since its value at $0$ is negative and its value at $1/10$ is $\frac{169}{100}>0$, that root lies below $\frac{1}{10}$. Therefore the cubic is positive for every $\alpha\geqslant \frac{1}{10}$. On the interval $[\frac{2}{n},\frac{1}{10}]$, we have
$g(\alpha)\leqslant  g\left(\frac{2}{n}\right)<n$
for $n\geqslant 14$. On the interval $\left[1/10,1\right]$ we have $g(\alpha)<20\leqslant n$.

Hence if $n\geqslant 20$ and $c\geqslant \frac{1}{2}$, then
$$\ln(Q(G))\geqslant L_n(1/2,\delta)>\frac{1}{2}\left(1+\frac{5}{n}-\frac{3}{n^2}\right)>\ln(Q_n).$$

To prove the case $c\leqslant \frac{1}{2}$ we will need to slightly improve on our bounds. We will show that if $n\geqslant 20$ and $h\geqslant \frac{n}{4}+\frac{9}{2}$, then $Q(G)>Q_n$. 
\medskip

In this part, we will use the average
$$\frac{1}{F_2(G)}\sum_{S \in \mathcal{F}_2(G)}\delta(S)\ln \left(1+\frac{1}{\delta(S)}\right)\ \ \ \mbox{instead of}\ \ \ \delta \ln\left(1+\frac{1}{\delta}\right)$$
in Theorem~\ref{ratio-minimum}.

Note that we have
\begin{equation} \label{delta-average}
\frac{1}{T(G)}\sum_{S \in \mathcal{F}_2(G)}\delta(S)=n-1
\end{equation}
by double-counting the pairs $(T,e)$, where $T$ is a spanning tree and $e\in E(T)$ and noting that $T-e$ is a $2$-forest with $e$ being in the cut of $T-e$.

We also need a slightly more involved identity. Note that by Lemma~\ref{effective-combinatorial} we have
$$|\left\{S\in \cF_{2}(G)\ |\ u,v\ \text{lie in different components of}\ S\right\}|=T(G)R_{\eff}(u,v).$$
Thus by double-counting, if $S_+$ and $S_-$ denote the vertex sets of the connected components of a $2$-forest $S$, then we have
\begin{equation} \label{product-average}
\frac{1}{T(G)}\sum_{S \in \mathcal{F}_2(G)}|S_+||S_-|=\sum_{u<v}R_{\eff}(u,v)=n\sum_{i=1}^{n-1}\frac{1}{\lambda_i},
\end{equation}
where $\lambda_1\geq \dots \geq \lambda_{n-1}\geq \lambda_n=0$ are the eigenvalues of the Laplacian matrix of the graph $G$. The second equality is a well-known identity about the Kirchhoff index, see for example \cite{xiao2003resistance}. Recall that $\sum_{i=1}^{n-1}\lambda_i=\mathrm{Tr}(L(G))=\sum_{i=1}^nd_i=2e(G)$.
Another notable fact is that the eigenvalues of the Laplacian matrix of the complement graph $\overline{G}$ are $n-\lambda_{n-1}\geq n-\lambda_{n-2}\geq \dots \geq n-\lambda_{1}\geq 0$, see for instance Lemma 13.1.3 in the book \cite{godsil2001algebraic}. This also implies that for any simple graph $H$ we have $\lambda_{\max}(H)\leqslant v(H)$ since the Laplacian eigenvalues of the complement of the graph are non-negative.
This implies that if $H_1,\dots ,H_k$ are the connected components of the graph $\overline{G}$, then $\lambda_1(\overline{G})\leqslant \max_{j=1}^kv(H_j)\leq e(\overline{G})+1=h+1$. Hence $\lambda_{n-1}=n-\lambda_1(\overline{G})\geqslant n-h-1$. Then by combining \eqref{delta-average} and \eqref{product-average} with the lower bound on $\lambda_{n-1}$ we have 
\begin{equation} \label{difference-upper-bound}
\frac{1}{T(G)}\sum_{S \in \mathcal{F}_2(G)}(|S_+||S_-|-\delta(S))=\sum_{i=1}^{n-1}\frac{n-\lambda_i}{\lambda_i}\leqslant \frac{1}{n-h-1}\sum_{i=1}^{n-1}(n-\lambda_i)=\frac{2h}{n-h-1}.
\end{equation}
(Another useful lower bound for $\lambda_{n-1}$ is in terms of the maximum degree $\Delta(\overline{G})$ of the complement:
$\lambda_{n-1}=n-\overline{\lambda_1}\geqslant n-2\Delta(\overline{G})=n-2(n-1-\delta)=2\delta-n+2$, but we will not use it.)
Hence we have

\begin{align*}
\ln(Q(G))&=\ln \left(\frac{F(G)}{T(G)}\right)\\
&\geqslant\frac{F_2(G)}{T(G)}\cdot \frac{1}{F_2(G)}\sum_{S\in \mathcal{F}_2(G)}\delta(S)\ln\left(1+\frac{1}{\delta(S)}\right) && (\text{Theorem~\ref{ratio-minimum}})\\
&=\frac{1}{T(G)}\sum_{S\in \mathcal{F}_2(G)}\delta(S)\ln\left(1+\frac{1}{\delta(S)}\right)\\
&\geqslant \frac{1}{T(G)}\sum_{S\in \mathcal{F}_2(G)}\left(1-\frac{1}{2\delta(S)}\right) &&(\text{Ineq.~\eqref{func-ineq-0}})\\
&=\frac{F_2(G)}{T(G)}-\frac{1}{2T(G)}\sum_{S\in \mathcal{F}_2(G)}\frac{1}{\delta(S)}\\
&=\frac{F_2(G)}{T(G)}-\frac{1}{2T(G)}\sum_{S\in \mathcal{F}_2(G)}\left(\frac{1}{\delta(S)}-\frac{1}{|S_+||S_-|}\right)\\
&\quad -\frac{1}{2T(G)}\sum_{S\in \mathcal{F}_2(G)}\frac{1}{|S_+||S_-|}\\
&\geqslant\frac{F_2(G)}{T(G)}-\frac{1}{2T(G)}\sum_{S\in \mathcal{F}_2(G)}\left(\frac{1}{\delta(S)}-\frac{1}{|S_+||S_-|}\right)\\
&\quad -\frac{1}{2T(G)}\sum_{S\in \mathcal{F}_2(G)}\frac{1}{n-1} && (|S_+||S_-|\geqslant n-1)\\
&=\frac{F_2(G)}{T(G)}\left(1-\frac{1}{2(n-1)}\right)-\frac{1}{2T(G)}\sum_{S\in \mathcal{F}_2(G)}\frac{|S_+||S_-|-\delta(S)}{\delta(S)|S_+||S_-|} \\
&\geqslant \frac{F_2(G)}{T(G)}\left(1-\frac{1}{2(n-1)}\right)-\\
&\quad -\frac{1}{2\delta (n-1)}\frac{1}{T(G)}\sum_{S\in \mathcal{F}_2(G)}(|S_+||S_-|-\delta(S)) && (\delta(S)|S_+||S_-|\geqslant \delta(n-1))\\
&=\frac{F_2(G)}{T(G)}\left(1-\frac{1}{2(n-1)}\right)-\frac{1}{2\delta (n-1)}\sum_{i=1}^{n-1}\frac{n-\lambda_i}{\lambda_i} && (\eqref{delta-average}\ \text{and}\ \eqref{product-average})\\
&\geqslant \frac{F_2(G)}{T(G)}\left(1-\frac{1}{2(n-1)}\right)-\frac{1}{2\delta (n-1)}\cdot \frac{2h}{n-h-1} && \eqref{difference-upper-bound}\\
&\geqslant \frac{F_2(G)}{T(G)}\left(1-\frac{1}{2(n-1)}\right)-\frac{1}{2(n-h-1)(n-1)}\cdot \frac{2h}{n-h-1} &&   (\delta\geqslant n-h-1)\\
&= \frac{F_2(G)}{T(G)}\left(1-\frac{1}{2(n-1)}\right)-\frac{h}{(n-h-1)^2(n-1)}  \\
&\geqslant \frac{F_2(G)}{T(G)}\left(1-\frac{1}{2(n-1)}\right)-\frac{n/2}{(n/2-1)^2(n-1)} && \left(h\leqslant \frac{n}{2}\right) \\
&\geqslant\frac{F_2(G)}{T(G)}\left(1-\frac{1}{2(n-1)}\right)-\frac{3}{n^2}. && (n\geqslant 14)\\
\end{align*}
Since $\ln(Q_n)<\frac{1}{2}+\frac{5}{2n}-\frac{3}{2n^2}$ it is enough to prove that
$$\frac{F_2(G)}{T(G)}\left(1-\frac{1}{2(n-1)}\right)\geqslant \frac{1}{2}+\frac{5}{2n}+\frac{3}{2n^2}.$$
Since 
$$\frac{F_2(G)}{T(G)}\geqslant \left(1+\frac{6}{n}\right)\frac{(n-1)^2}{4m}$$
this is true if
\begin{equation} \label{m-upper-bound}
m\leqslant \left(1+\frac{6}{n}\right)\frac{(n-1)^2}{4}\left(1-\frac{1}{2(n-1)}\right)\left(\frac{1}{2}+\frac{5}{2n}+\frac{3}{2n^2}\right)^{-1}=\frac{2n^4 + 7n^3 - 27n^2 + 18n}{4(n^2 + 5n + 3)}.
\end{equation}
We have
$$\frac{2n^4 + 7n^3 - 27n^2 + 18n}{4(n^2 + 5n + 3)}=\binom{n}{2}-\frac{n}{4}-\frac{9(2n^2 - 3n)}{4(n^2 + 5n + 3)}>\binom{n}{2}-\frac{n}{4}-\frac{9}{2}.$$
Hence if $m\leqslant \binom{n}{2}-\frac{n}{4}-\frac{9}{2}$, equivalently $h\geqslant \frac{n}{4}+\frac{9}{2}$, then inequality \eqref{m-upper-bound} is also satisfied.

For $n\geqslant 21$ we have $\lfloor 0.49n+1/2\rfloor>\frac{n}{4}+\frac{9}{2}$. This completes the proof of Theorem~\ref{forest-tree-ratio-2-quantitative}. We also have $\left(\frac{1}{4}+\varepsilon\right)n\geqslant \frac{n}{4}+\frac{9}{2} $
whenever $n\geqslant \frac{9}{2\varepsilon}$. Thus, we can take $n_0(\varepsilon) = \max\left\{ 20,\left\lceil\frac{9}{2\varepsilon}\right\rceil \right\}$
to finish the proof of Theorem~\ref{forest-tree-ratio-2}. 
 
 \end{proof}

\section{Inclusion-exclusion for very dense graphs}
\label{incusion-exclusion-sect}

In this section, we prove Theorem~\ref{forest-tree-ratio-3} together with the following quantitative version.

\begin{Th} \label{forest-tree-ratio-3-quantitative}
Let $n\geqslant 10$. Let $G$ be a connected graph on $n$ vertices with $m=\binom{n}{2}-h$ edges, where $h\leqslant 0.49n$. Then $Q(G)\geqslant Q(K_n)$ with equality if and only if $G=K_n$. 
\end{Th}

Let $M=E(K_n)\setminus E(G)$, where $|M|=h\leqslant \max\left(0.49n, \left(\frac{1}{2}-\varepsilon\right)n\right)$. For $S\subseteq M$, let $F_S$ and $T_S$ be the numbers of spanning forests and
spanning trees of $K_n$ that contain every edge of $S$.  If $S$ contains a
cycle, then $F_S=T_S=0$. If $S$ is a forest, then let $Q_S=\frac{F_S}{T_S}$, and $Q_n=\frac{F(K_n)}{T(K_n)}=Q_{\emptyset}$.

Inclusion--exclusion gives
\[
  F(G)=\sum_{S\subseteq M}(-1)^{|S|}F_S,
  \qquad
  T(G)=\sum_{S\subseteq M}(-1)^{|S|}T_S.
\]
Let $d(S)=Q_n-Q_S$. Later we will show that $d(S)\geqslant 0$. We have
\begin{align*}
F(G)-Q_nT(G)&=\sum_{S\subseteq M}(-1)^{|S|}(F_S-Q_nT_S)\\
&=\sum_{\substack{S\subseteq M,\\ S\text{ a forest}}}(-1)^{|S|}(F_S-Q_nT_S)\\
&=\sum_{\substack{S\subseteq M,\\ S\text{ a forest}}}(-1)^{|S|}T_S(Q_S-Q_n)\\
&=\sum_{\substack{S\subseteq M,\\ S\text{ a forest}}}(-1)^{|S|+1}T_Sd(S).
\end{align*}
Since $F(K_n)=Q_nT(K_n)$, the $S=\varnothing$ term cancels and we get that
\begin{equation}\label{eq:IE-main}
  F(G)-Q_nT(G)=\sum_{s=1}^h(-1)^{s+1}D_s,
\end{equation}
where
\begin{equation}\label{eq:Ds-def}
  D_s=
  \sum_{\substack{S\subseteq M,\ |S|=s\\S\text{ a forest}}}
  T_Sd(S).
\end{equation}
We will show that the positive terms $D_s$ are strictly decreasing as long as they are nonzero. Since
$$F(G)-Q_nT(G)=(D_1-D_2)+(D_3-D_4)+\dots$$
this will immediately give that $F(G)-Q_nT(G)\geqslant 0$ with equality if and only if $G=K_n$. To better understand $d(S)=Q_n-Q_S$, observe that $F_S$ and $T_S$ depend only on the component sizes of $S$. Indeed, the forests and trees of $K_n$ containing the set $S$ are in bijection with the trees and forests of the multigraph obtained from $K_n$ by contracting each component of $S$. Let
\[
  \boldsymbol a=(a_1,\dots,a_m),\qquad a_i\in\mathbb N,
  \qquad a_1+\cdots+a_m=n.
\]
Let $K(\boldsymbol a)$ be the weighted complete graph on $[m]$ in which the
edge $ij$ has weight $a_i a_j$.  Equivalently, it is the multigraph with
$a_i a_j$ parallel edges between $i$ and $j$.  Let $\tau(\boldsymbol a)$ and
$\phi(\boldsymbol a)$ be its weighted spanning-tree and spanning-forest enumerators, and set
\[
  Q(\boldsymbol a)=\frac{\phi(\boldsymbol a)}{\tau(\boldsymbol a)}.
\]
Let $\boldsymbol a(S)$ denote the vector of component sizes of $S$. Then $F_S=\phi(\boldsymbol a(S))$, $T_S=\tau(\boldsymbol a(S))$ and $Q_S=Q(\boldsymbol a(S))$.

By the Cayley-Hurwitz formula we have
\begin{equation}
\tau(\boldsymbol a)=n^{m-2}\prod_{i=1}^m a_i
\end{equation}
This implies that if $S$ is an $s$-edge forest and an edge $e\in M\setminus S$ joins two components of $S$, of sizes $a$ and $b$, then
\begin{equation}\label{eq:tree-contraction-ratio}
  \frac{T_{S+e}}{T_S}
  =\frac{a+b}{nab}
  =\frac{1}{n}\left(\frac{1}{a}+\frac{1}{b}\right)
  \leqslant\frac{2}{n}.
\end{equation}
Later we will show that if $|S|\leqslant\frac{n}{2}-1$, then $d(S+e)\leqslant \gamma d(S)$, where $\gamma=3$ for $|S|\geqslant 2$, and  $\gamma=2.04$ for $|S|=1$. In the case $|S|=1$ we can also take $\gamma=2+\varepsilon$ if $n$ is large enough. Using the fact that every $(s+1)$-edge forest contains exactly $s+1$ forests of size $s$, this implies that
\begin{align*}
  (s+1)D_{s+1}
  &=\sum_{\substack{|S|=s,\ S\text{ a forest}}}
    \ \sum_{\substack{e\in M\setminus S\\S+e\text{ a forest}}}
      T_{S+e}d(S+e)\\
&\leqslant \sum_{\substack{|S|=s,\ S\text{ a forest}}}
    \ \sum_{\substack{e\in M\setminus S\\S+e\text{ a forest}}}
      \frac{2}{n}T_{S}\cdot \gamma d(S)\\
&\leqslant\frac{2\gamma}{n}
    \sum_{\substack{|S|=s,\ S\text{ a forest}}}
      (h-s)T_Sd(S)\\
&=\frac{2\gamma}{n}(h-s)D_s.
\end{align*}
Therefore, for $s\geqslant 1$ we get that
\begin{equation}\label{eq:Ds-ratio}
  \frac{D_{s+1}}{D_s}
  \leqslant\frac{2\gamma(h-s)}{n(s+1)}
\end{equation}
This means that if $s\geqslant 5$ and $h\leqslant n+5$, then $D_s\geqslant D_{s+1}$.
In order to have $D_3\geq D_4$ we need $h\leqslant \frac{2}{3}n+3$, but the real bottleneck is clearly the case $D_1\geqslant D_2$. Since in this case we have a better $\gamma$ we get that $\frac{D_2}{D_1}<1$ if $h<\frac{1}{2.04}n+1$.

In summary, to prove that $Q(G)\geqslant Q(K_n)$ for a graph $G$ with more than $\binom{n}{2}-(\frac{1}{2}-\varepsilon)n$ edges, it is enough to prove that (i) $d(S)\geqslant 0$ for all $S$, (ii) $d(S+e)\leqslant 3d(S)$ whenever $2\leqslant |S|\leqslant \frac{n}{2}-1$, and $d(S+e)\leqslant (2+\varepsilon)d(S)$ whenever $|S|=1$. Since we understand $\tau(\boldsymbol a)$ rather well, the problem essentially boils down to the analysis of $\phi(\boldsymbol a)$. This is the content of the next section.

\subsection{Weighted complete graphs}

As before, for
$$\boldsymbol a=(a_1,\dots,a_m),\qquad a_i\in\mathbb N,
  \qquad a_1+\cdots+a_m=n,
$$
let $K(\boldsymbol a)$ be the weighted complete graph on $[m]$ in which the
edge $ij$ has weight $a_i a_j$.  Equivalently, it is the multigraph with
$a_i a_j$ parallel edges between $i$ and $j$.  Recall that $\tau(\boldsymbol a)$ and
$\phi(\boldsymbol a)$ denote its weighted spanning-tree and spanning-forest
enumerators, and set
\[
  Q(\boldsymbol a)=\frac{\phi(\boldsymbol a)}{\tau(\boldsymbol a)}.
\]

Let $\He_r(x)$ denote the probabilists' Hermite polynomials:
\[
  \He_{-1}(x)=0,\qquad \He_0(x)=1,\qquad
  \He_{r+1}(x)=x\He_r(x)-r\He_{r-1}(x)\ \ (r\geqslant 0).
\]
The first few Hermite polynomials are
$$\He_1(x)=x,\ \ \He_2(x)=x^2-1,\ \ \He_3(x)=x^3-3x,\ \ \He_4(x)=x^4-6x^2+3,$$
$$\He_5(x)=x^5-10x^3+15x,\ \ \He_6(x)=x^6-15x^4+45x^2-15.$$
The polynomial $\He_n(x)$ is also the matching polynomial of the complete graph $K_n$.

\begin{Lemma}[Weighted Cayley--Hurwitz formulas]\label{lem:weighted-formulas}
For $\boldsymbol a$ as above,
\begin{align}
  \tau(\boldsymbol a)
    &=n^{m-2}\prod_{i=1}^m a_i, \label{eq:weighted-tree}\\
  \phi(\boldsymbol a)
    &=1+\sum_{j=2}^m \He_{j-2}(n)e_j(\boldsymbol a),
      \label{eq:weighted-forest}
\end{align}
where $e_j$ is the $j$th elementary symmetric polynomial.
Consequently, with $x_i=a_i^{-1}$,
\begin{equation}\label{eq:Q-formula}
  Q(\boldsymbol a)
  =n^{2-m}\left(
      e_m(\boldsymbol x)
      +\sum_{k=0}^{m-2}\He_{m-k-2}(n)e_k(\boldsymbol x)
    \right).
\end{equation}
\end{Lemma}

\begin{proof}
For a tree $T$ on $[m]$,
\[
  \prod_{ij\in E(T)}a_i a_j=\prod_{i=1}^m a_i^{\deg_T(i)}.
\]
Summing by Pr\"ufer codes gives
\[
  \sum_T\prod_{i=1}^m a_i^{\deg_T(i)}
  =\left(\prod_{i=1}^m a_i\right)
    \left(\sum_{i=1}^m a_i\right)^{m-2},
\]
which proves \eqref{eq:weighted-tree}.

For \eqref{eq:weighted-forest}, partition a forest into its tree components.
If $B\subseteq[m]$ is a component, weighted Cayley gives
\[
  \sum_{\substack{T\text{ a tree}\\V(T)=B}}
       \prod_{ij\in E(T)}a_i a_j
  =\left(\sum_{i\in B}a_i\right)^{|B|-2}
     \prod_{i\in B}a_i,
\]
with the natural value $1$ when $|B|=1$.  Hence
\begin{equation}\label{eq:partition-sum}
  \phi(\boldsymbol a)
  =\sum_{\pi\in\Pi([m])}
      \prod_{B\in\pi}
      \left(\sum_{i\in B}a_i\right)^{|B|-2}
      \prod_{i\in B}a_i.
\end{equation}
We now evaluate the partition sum in full detail.  Introduce commuting
indeterminates $z_1,\dots,z_m$.  For $A\subseteq[m]$, write
\[
  z_A=\prod_{i\in A}z_i,\qquad
  a_A=\sum_{i\in A}a_i,\qquad
  p_A=\prod_{i\in A}a_i.
\]
Let $S=S(\boldsymbol z)$ be the unique formal power series with zero constant
term satisfying
\begin{equation}\label{eq:rooted-tree-functional}
  S=\sum_{i=1}^m a_i z_i\exp(a_iS),
\end{equation}
and put
\[
  R_i=z_i\exp(a_iS).
\]
The squarefree coefficient $[z_A]R_i$, for $i\in A$, is the total weight of
trees on $A$ rooted at $i$.  Indeed, after deleting the root $i$, a rooted
tree becomes an unordered family of rooted trees attached to $i$.  If one of
those subtrees is rooted at $j$, joining it to $i$ contributes the edge
weight $a_i a_j$.  Hence the generating series for one possible branch at
$i$ is
\[
  \sum_{j=1}^m a_i a_j R_j=a_iS,
\]
and the exponential in the definition of $R_i$ forms an unordered family of
such branches.  Terms in which two branches use a common label contain some
$z_j^2$ and therefore do not contribute to squarefree coefficients.  This
proves the claim about $R_i$.

Set
\begin{equation}\label{eq:unrooted-tree-series}
  \mathcal U=\sum_{i=1}^m R_i-\frac{S^2}{2}.
\end{equation}
Its squarefree coefficient $[z_A]\mathcal U$ is the total weight of
unrooted trees on $A$.  To see this, the first term in
\eqref{eq:unrooted-tree-series} counts every unrooted tree on $A$ once for
each of its $|A|$ possible distinguished vertices.  On the other hand,
$S^2/2$ counts it once for each distinguished edge: cutting that edge gives
an unordered pair of rooted trees, and the two factors $a_iR_i$ and $a_jR_j$
provide precisely the missing edge weight $a_i a_j$.  Since a tree on $A$
has $|A|-1$ edges, the difference counts it exactly once.

A forest is an unordered set of tree components.  Consequently,
\begin{equation}\label{eq:forest-coefficient}
  \phi(\boldsymbol a)=[z_{[m]}]\exp(\mathcal U).
\end{equation}
We next compute this coefficient.  Let
\[
  B(t)=\sum_{i=1}^m z_i\exp(a_it),
\]
so that $\sum_iR_i=B(S)$.  Since only squarefree coefficients are relevant,
we may work modulo the ideal $(z_1^2,\dots,z_m^2)$.  Each
$z_i\exp(a_iS)$ has a factor $z_i$, and hence
\begin{align}
  \exp(\mathcal U)
  &=\exp\left(B(S)-\frac{S^2}{2}\right) \notag\\
  &\equiv
    \exp\left(-\frac{S^2}{2}\right)
    \prod_{i=1}^m\left(1+z_i\exp(a_iS)\right) \notag\\
  &=\sum_{J\subseteq[m]}
      z_J\exp\left(a_JS-\frac{S^2}{2}\right).
      \label{eq:forest-squarefree-expansion}
\end{align}
For $K\subseteq[m]$, let $S_K$ be the unique solution with zero constant term
of
\begin{equation}\label{eq:restricted-functional}
  S_K=\sum_{i\in K}a_i z_i\exp(a_iS_K).
\end{equation}
In the term of \eqref{eq:forest-squarefree-expansion} indexed by
$J=[m]\setminus K$, the factor $z_J$ already supplies all variables whose
indices lie in $J$.  Thus any further occurrence of one of these variables
cannot contribute to $z_{[m]}$; equivalently, we set $z_j=0$ for $j\in J$.
By uniqueness in \eqref{eq:rooted-tree-functional}, this changes $S$ into
$S_K$.  Therefore \eqref{eq:forest-coefficient} becomes
\begin{equation}\label{eq:forest-lagrange-start}
  \phi(\boldsymbol a)
  =\sum_{K\subseteq[m]}
     [z_K]\exp\left(a_{K^c}S_K-\frac{S_K^2}{2}\right).
\end{equation}

We evaluate the summands by the one-variable Lagrange--Buerhmann formula.
Fix a nonempty $K$ of size $k$, and put
\[
  \Psi_K(t)=\sum_{i\in K}a_i z_i\exp(a_it),
\]
so that $S_K=\Psi_K(S_K)$.  For any formal power series $f$, Lagrange
inversion gives
\[
  f(S_K)=f(0)+\sum_{r\geq1}\frac1r
       [t^{r-1}]f'(t)\Psi_K(t)^r.
\]
The monomial $z_K$ has total degree $k$, so only the term $r=k$ can
contribute to its coefficient.  Moreover,
\[
  [z_K]\Psi_K(t)^k
  =k!\,p_K\exp(a_Kt),
\]
because the $k$ factors must select each index in $K$ exactly once.  Hence
\begin{equation}\label{eq:squarefree-lagrange}
  [z_K]f(S_K)
  =(k-1)!\,p_K[t^{k-1}]f'(t)\exp(a_Kt).
\end{equation}
Apply this with
\[
  f(t)=\exp\left(a_{K^c}t-\frac{t^2}{2}\right).
\]
Since $a_K+a_{K^c}=n$, equation \eqref{eq:squarefree-lagrange} yields
\begin{align*}
 &[z_K]\exp\left(a_{K^c}S_K-\frac{S_K^2}{2}\right)\\
 &\quad=(k-1)!\,p_K[t^{k-1}]
       (a_{K^c}-t)\exp\left(nt-\frac{t^2}{2}\right).
\end{align*}
The exponential generating function of the probabilists' Hermite
polynomials is
\[
  \exp\left(nt-\frac{t^2}{2}\right)
  =\sum_{r\geq0}\He_r(n)\frac{t^r}{r!}.
\]
Consequently,
\begin{equation}\label{eq:subset-contribution}
  [z_K]\exp\left(a_{K^c}S_K-\frac{S_K^2}{2}\right)
  =p_K\left(
      a_{K^c}\He_{k-1}(n)
      -(k-1)\He_{k-2}(n)
    \right),
\end{equation}
where the second term is understood as $0$ when $k=1$.

The term $K=\varnothing$ in \eqref{eq:forest-lagrange-start} equals $1$.
Summing \eqref{eq:subset-contribution} over all $K$ of a fixed size $k$
and using
\begin{equation}\label{eq:subset-elementary-identity}
  \sum_{\substack{K\subseteq[m]\\|K|=k}}p_Ka_{K^c}
  =(k+1)e_{k+1}(\boldsymbol a)
\end{equation}
gives
\begin{equation*}
  \phi(\boldsymbol a)
  =1+\sum_{k=1}^m\Bigl(
       (k+1)\He_{k-1}(n)e_{k+1}(\boldsymbol a)
       -(k-1)\He_{k-2}(n)e_k(\boldsymbol a)
     \Bigr),
\end{equation*}
with $e_{m+1}=0$.  Identity
\eqref{eq:subset-elementary-identity} is immediate: after expanding
$a_{K^c}=\sum_{i\notin K}a_i$, each $(k+1)$-subset contributes its product
once for each of its $k+1$ elements.  Finally, reindexing the first sum gives
\begin{equation*}
  \phi(\boldsymbol a)
  =1+\sum_{j=2}^m
       \left(j-(j-1)\right)\He_{j-2}(n)e_j(\boldsymbol a)
  =1+\sum_{j=2}^m\He_{j-2}(n)e_j(\boldsymbol a),
\end{equation*}
which is \eqref{eq:weighted-forest}.
Finally, divide \eqref{eq:weighted-forest} by
\eqref{eq:weighted-tree} and use
$\frac{e_j(\boldsymbol a)}{\prod_i a_i}
  =e_{m-j}(\boldsymbol x)$
to obtain \eqref{eq:Q-formula}.
\end{proof}

For the rest of the proof, write
\[
  H_r=\He_r(n),\qquad h_r=\frac{H_r}{n^r}.
\]

\begin{Lemma}\label{lem:Hermite-bounds}
For $1\leqslant r\leqslant n$,
\begin{equation}\label{eq:Hermite-ratio}
  n-1\leqslant \frac{H_r}{H_{r-1}}\leqslant n.
\end{equation}
Consequently,
\begin{equation}\label{eq:h-monotone}
  1=h_0=h_1\geqslant h_2\geqslant\cdots>0,
\end{equation}
and
\begin{equation}\label{eq:h-exponential}
  h_r\leq
  \prod_{j=1}^{r-1}\left(1-\frac{j}{n^2}\right)
  \leqslant \exp\left(-\frac{r(r-1)}{2n^2}\right).
\end{equation}
\end{Lemma}

\begin{proof}
Let $R_r=H_r/H_{r-1}$.  The Hermite recurrence gives
\[
  R_r=n-\frac{r-1}{R_{r-1}}.
\]
Starting from $R_1=n$, induction shows that $n-1\leqslant R_r\leqslant n$ for
$1\leqslant r\leqslant n$.  This proves \eqref{eq:Hermite-ratio} and
\eqref{eq:h-monotone}.  Moreover,
\[
  h_r=h_{r-1}-\frac{r-1}{n^2}h_{r-2}
  \leqslant \left(1-\frac{r-1}{n^2}\right)h_{r-1},
\]
because $h_{r-2}\geqslant h_{r-1}$.  Iteration, followed by
$1-x\leqslant \ee^{-x}$, proves \eqref{eq:h-exponential}.
\end{proof}

Recall that a function $f:\mathbb{R}^d\to \mathbb{R}$ is Schur-convex if $f(\boldsymbol x)\geqslant f(\boldsymbol y)$ whenever $\boldsymbol x$ majorizes $\boldsymbol y$. The Schur-Ostrowski criterion says that if $f$ is symmetric and all first partial derivatives exist, then $f$ is Schur-convex if and only if for all $1\leqslant i,j\leqslant d$ we have
$$(x_i-x_j)\left(\frac{\partial f}{\partial x_i}-\frac{\partial f}{\partial x_j}\right)\geq 0$$
for all $\boldsymbol x \in \mathbb{R}^d$.

\begin{Lemma}\label{lem:Schur-convex}
For fixed $m$ and $n$, the function $Q(a_1,\dots,a_m)$ is Schur-convex on the
positive orthant intersected with $a_1+\cdots+a_m=n$.
\end{Lemma}

\begin{proof}
It is enough, by \eqref{eq:Q-formula} and Lemma~\ref{lem:Hermite-bounds}, to
show that
\[
  (a_1,\dots,a_m)\longmapsto
  e_k(a_1^{-1},\dots,a_m^{-1})
\]
is Schur-convex.  Put $x_i=a_i^{-1}$, and let $E_r$ denote the $r$th
elementary symmetric polynomial in all variables except $x_i,x_j$ with the convention that $E_0=1$ and $E_j=0$ if $j<0$ or $j>m-2$.  Then
\[
  \frac{\partial e_k(\boldsymbol x)}{\partial a_i}
  =-x_i^2(E_{k-1}+x_jE_{k-2}).
\]
It follows that
\begin{equation*}
 (a_i-a_j)
 \left(
   \frac{\partial e_k(\boldsymbol x)}{\partial a_i}
   -\frac{\partial e_k(\boldsymbol x)}{\partial a_j}
 \right) 
 =
 \frac{(x_i-x_j)^2}{x_ix_j}
 \left((x_i+x_j)E_{k-1}+x_ix_jE_{k-2}\right)\geqslant0.
\end{equation*}
The Schur--Ostrowski criterion completes the proof.
\end{proof}

For $0\leqslant t\leqslant n-1$, define
\[
  P_t=(t+1,1^{\,n-t-1}).
\]
A positive integral vector of length $n-t$ and total sum $n$ is majorized by
$P_t$.  Lemma~\ref{lem:Schur-convex} therefore gives
\begin{equation}\label{eq:majorization-bound}
  Q(\boldsymbol a)\leqslant Q(P_t)
\end{equation}
for every such vector $\boldsymbol a$.

We next compare $Q$ before and after merging two coordinates.

\begin{Lemma}[Exact merge formula]\label{lem:merge-formula}
Let
\[
  s=a+b,\qquad u=\frac1a+\frac1b,
\]
and let $\boldsymbol c=(c_1,\dots,c_\ell)$ satisfy
$a+b+c_1+\cdots+c_\ell=n$.  Put
$$E_j=e_j(c_1^{-1},\dots,c_\ell^{-1}),$$
with the convention that $E_0=1$ and $E_j=0$ if $j<0$ or $j>\ell$.
If $\ell\geqslant 1$, then, with $H_{-1}=0$,
\begin{equation}
 n^\ell\bigl(Q(a,b,\boldsymbol c)-Q(s,\boldsymbol c)\bigr)=
 \sum_{j=0}^{\ell-1}
 \left[
   uH_{\ell-j-1}
   -\left(\ell-j-1+\frac{n-u}{s}\right)H_{\ell-j-2}
 \right]E_j
+\left(1-\frac{n-u}{s}\right)E_\ell.
 \label{eq:merge-formula}
\end{equation}
If $\ell=0$, then $Q(a,b)-Q(s)=\frac{1}{ab}$.
\end{Lemma}

\begin{proof}
We have 
\[
  e_k(a^{-1},b^{-1},\boldsymbol c^{-1})
  =E_k+uE_{k-1}+\frac{u}{s}E_{k-2}
\ \ \text{and}\ \ 
  e_k(s^{-1},\boldsymbol c^{-1})
  =E_k+\frac{1}{s}E_{k-1}.
\]
Now let us use \eqref{eq:Q-formula} with $m=\ell+2$. If $\ell\geqslant 1$ we have
\begin{align*}
 n^\ell\bigl(Q(a,b,\boldsymbol c)-Q(s,\boldsymbol c)\bigr)&=e_m(a^{-1},b^{-1},\boldsymbol c^{-1})+\sum_{j=0}^{m-2}H_{m-j-2}e_j(a^{-1},b^{-1},\boldsymbol c^{-1})\\
 &\quad -n e_{m-1}(s^{-1},\boldsymbol c^{-1})-\sum_{j=0}^{m-3}nH_{m-j-3}e_j(s^{-1},\boldsymbol c^{-1})\\
 &=\frac{1}{ab}E_{\ell}(\boldsymbol c^{-1})+\sum_{j=0}^{\ell}H_{\ell-j}(E_j+uE_{j-1}+\frac{u}{s}E_{j-2})\\
&\quad -ns^{-1}E_{\ell}(\boldsymbol c^{-1})-\sum_{j=0}^{\ell-1}nH_{\ell-1-j}(E_j+s^{-1}E_{j-1})\\
&=((ab)^{-1}-ns^{-1}+1)E_{\ell}\\
&\quad +\sum_{j=0}^{\ell-1}E_j\left(H_{\ell-j}+uH_{\ell-j-1}+\frac{u}{s}H_{\ell-j-2}-nH_{\ell-1-j}-\frac{n}{s}H_{\ell-j-2}\right)\\
&=\sum_{j=0}^{\ell-1}E_j\left(uH_{\ell-j-1}-\left(\ell-j-1+\frac{n-u}{s}\right)H_{\ell-j-2}\right)+\left(1+\frac{u-n}{s}\right)E_{\ell},
\end{align*}
where after collecting the coefficient of each $E_j$, we used the Hermite recurrence
$H_{r+1}=nH_r-rH_{r-1}$ for $r=\ell-j-1$. This gives exactly \eqref{eq:merge-formula}. If $\ell=0$, then we can directly compute
$$Q(a,b)-Q(s)=\frac{ab+1}{ab}-1=\frac{1}{ab},$$
this is also equal to $\left(1-\frac{n-u}{s}\right)E_0$. Indeed, $s=n$, $E_0=1$, and $\frac{u}{s}=\frac{1}{ab}$.
\end{proof}

\begin{Lemma}\label{lem:P-monotone}
The sequence $Q(P_t)$ is strictly decreasing:
\[
  Q(P_0)>Q(P_1)>\cdots>Q(P_{n-1}).
\]
\end{Lemma}

\begin{proof}
For $2\leqslant s\leqslant n$, compare $P_{s-2}$ and $P_{s-1}$ by applying
Lemma~\ref{lem:merge-formula} with
\[
  a=s-1,\qquad b=1,\qquad \boldsymbol c=1^{\,n-s},
  \qquad u=\frac{s}{s-1}, \qquad \ell=n-s
\]
Let us first check the case when $\ell\geqslant 1$.
By $\frac{H_r}{H_{r-1}}\geqslant n-1$, the coefficient of $E_j$ for $r=\ell-j-1\geqslant 1$ in
\eqref{eq:merge-formula} satisfies
\begin{align*}
uH_{\ell-j-1}-\left(\ell-j-1+\frac{n-u}{s}\right)H_{\ell-j-2}
&\geqslant \left(u(n-1)-(\ell-j-1)-\frac{n-u}{s}\right)H_{\ell-j-2}\\
&\geqslant \left(u(n-1)-(n-s-1)-\frac{n-u}{s}\right)H_{\ell-j-2}\\
&=\left(\frac{s}{s-1}(n-1)-(n-s-1)-\frac{n-\frac{s}{s-1}}{s}\right)H_{\ell-j-2}\\
&=\left(\frac{n}{s(s-1)}+s\right)H_{\ell-j-2}\\
&>0
\end{align*}

Using $E_{\ell}=1$ and $E_{\ell-1}=n-s$, the remaining $r=0$ term, when present, together with the final term in
\eqref{eq:merge-formula}, equals
\begin{align*}
uE_{\ell-1}+\left(1-\frac{n-u}{s}\right)E_{\ell}&=u(n-s)+\left(1-\frac{n-u}{s}\right)\\
&=\frac{s}{s-1}(n-s)+\left(1-\frac{n-\frac{s}{s-1}}{s}\right)\\
&=\left(\frac{s}{s-1}-\frac{1}{s}\right)n-\frac{s^2}{s-1}+1+\frac{1}{s-1}\\
&=\frac{s^2-s+1}{s(s-1)}n-s\\
&\geqslant \frac{s^2-s+1}{s(s-1)}s-s\\
&=\frac{1}{s-1}\\
&>0.
\end{align*}
Finally, when $\ell=0$ we can directly check that 
$$Q(P_{n-2})-Q(P_{n-1})=Q(n-1,1)-Q(n)=\frac{n}{n-1}-1=\frac{1}{n-1}>0.$$
Thus, every merge of a
singleton into the nontrivial part strictly decreases $Q$.
\end{proof}

Note that
\[
  Q_n=Q(1,\dots,1)=\frac{F(K_n)}{T(K_n)}.
\]
Set
\begin{equation}\label{eq:delta-def}
  \delta_n=Q_n-Q(2,1^{\,n-2}).
\end{equation}
Combining \eqref{eq:majorization-bound} with
Lemma~\ref{lem:P-monotone} gives the following useful uniform estimate.

\begin{Cor}[Minimum contraction defect]\label{cor:min-defect}
If $\boldsymbol a\neq(1,\dots,1)$ is a positive integral vector of total mass
$n$, then
\begin{equation}\label{eq:min-defect}
  Q_n-Q(\boldsymbol a)\geqslant\delta_n>0.
\end{equation}
\end{Cor}

We shall also need a quantitative lower bound on $\delta_n$.  Let $f_{n,k}$
be the number of spanning forests of $K_n$ with $k$ components, and put
$T_n=n^{n-2}$.

\begin{Lemma}\label{lem:delta-lower}
For every $n\geqslant3$,
\begin{equation}\label{eq:delta-sum}
  \delta_n
  =\sum_{k=2}^n\frac{f_{n,k}}{T_n}\frac{k-1}{n-1},
\end{equation}
and
\begin{equation}\label{eq:delta-lower}
  \delta_n\geqslant
  \frac{3}{4n}+\frac{23n^2+34n-120}{4n^4}.
\end{equation}
\end{Lemma}

\begin{proof}
Let $\cF_{n,k}$ be the set of forests of $K_n$ with $k$ connected components, and let $\cT_n$ be the set of spanning trees of $K_n$. For a fixed edge $e$ of $K_n$,
\[
  \Pr(e\in\cF_{n,k})=\frac{n-k}{\binom{n}{2}},
  \qquad
  \Pr(e\in\cT_n)=\frac{n-1}{\binom{n}{2}}=\frac{2}{n}.
\]
The ratio $Q(2,1^{n-2})$ is the forest-to-tree ratio conditioned on containing $e$. Hence
$$\phi(2,1^{n-2})=\sum_{k=1}^n\frac{n-k}{\binom{n}{2}}f_{n,k}\ \ \ \text{and}\ \ \ \tau(2,1^{n-2})=\frac{n-1}{\binom{n}{2}}T_n,$$
consequently
$$\delta_n=Q_n-Q(2,1^{n-2})=\sum_{k=1}^n\frac{f_{n,k}}{T_n}-\sum_{k=1}^n\frac{\frac{n-k}{\binom{n}{2}}f_{n,k}}{\frac{n-1}{\binom{n}{2}}T_n}=\sum_{k=1}^n\frac{f_{n,k}}{T_n}\frac{k-1}{n-1}.$$
Since the $k=1$ term is $0$, this is exactly
\eqref{eq:delta-sum}.

R\'enyi \cite{renyi1959some} proved that (see also Lemma~\ref{2-forests-complete-graph} for the first identity)
\begin{align}
  \frac{f_{n,2}}{T_n}
    &=\frac{(n-1)(n+6)}{2n^2}, \label{eq:fn2}\\
  \frac{f_{n,3}}{T_n}
    &=\frac{(n-1)(n-2)(n^2+13n+60)}{8n^4}.
      \label{eq:fn3}
\end{align}
Keeping only the $k=2,3$ terms in \eqref{eq:delta-sum} gives
\begin{equation*}
  \delta_n
  \geqslant \frac{1}{n-1}
     \left(\frac{f_{n,2}}{T_n}
     +2\frac{f_{n,3}}{T_n}\right)
  =\frac{3}{4n}+\frac{23n^2+34n-120}{4n^4}.
\end{equation*}
\end{proof}

\begin{Rem}
By combining $f_{n,k}\sim \frac{1}{2^{k-1}(k-1)!}$ with some bound on the tail one can prove that $\delta_n=\frac{\sqrt{e}}{2n}+O\left(\frac{1}{n^2}\right)$.

\end{Rem}

Our next goal is a quantitative merge estimate.

\begin{Lemma}[Contraction lemma]\label{lem:contraction}
Suppose that $n\geqslant 10$, $t\geqslant 1$, and
$t+1\leqslant \frac{n}{2}$.
Let
\[
  \boldsymbol a=(a,b,c_1,\dots,c_\ell),
  \qquad \ell=n-t-2,
\]
have total mass $n$.  Then
\begin{equation}\label{eq:contraction-bound}
  Q(a,b,\boldsymbol c)-Q(a+b,\boldsymbol c)\leqslant 2\delta_n.
\end{equation}
\end{Lemma}

\begin{proof}
Set
\[
  \Delta=Q(a,b,\boldsymbol c)-Q(a+b,\boldsymbol c),
  \qquad s=a+b,
  \qquad u=\frac{1}{a}+\frac{1}{b}.
\]
The merged component has $s$ vertices, while the resulting contraction uses
$t+1$ edges; hence
\[
  s\leqslant t+2\leqslant \frac {n}{2}+1.
\]

For $r\geqslant0$, define
\[
  C_r(s,u)=uH_r-
  \left(r+\frac{n-u}{s}\right)H_{r-1}.
\]
With this notation and $r=\ell-j-1$ one can rewrite \eqref{eq:merge-formula} as
$$n^{\ell}\Delta=\sum_{j=0}^{\ell-1}C_{\ell-j-1}(s,u)E_j+\left(1-\frac{n-u}{s}\right)E_{\ell}.$$
First, observe that for $n\geqslant 5$ we have
\begin{equation}\label{eq:last-coeff-negative}
  1-\frac{n-u}{s}\leqslant 0.
\end{equation}
Indeed, since $u\leqslant \frac{s}{s-1}$ and $s\leqslant \frac{n}{2}+1$ we have
$$1-\frac{n-u}{s}\leqslant \frac{s}{s-1}-\frac{n}{s}=\frac{1}{s}\left(\frac{s^2}{s-1}-n\right)\leqslant \frac{1}{s}\left(\frac{2}{n}\left(\frac{n}{2}+1\right)^2-n\right)\leqslant 0$$
if $n\geqslant 5$, where we used the fact that the function $\frac{s^2}{s-1}$ is monotone increasing for $s\geqslant 2$, and that $\frac{n^2}{2}>\left(\frac{n}{2}+1\right)^2$ for $n\geqslant 5$. Thus, the last term in \eqref{eq:merge-formula} may be discarded when seeking an upper bound.

For fixed $s$, the function $C_r(s,u)$ is increasing in $u$, while $u\leq\frac{s}{s-1}$.
Furthermore,
\begin{equation} \label{eq:coefficient-comparison}
 C_r(2,2)-C_r\left(s,\frac{s}{s-1}\right) 
  =\frac{s-2}{s-1}
 \left[
   H_r-\left(\frac{n(s-1)}{2s}-1\right)H_{r-1}
 \right]\geqslant0,
 \end{equation}
by \eqref{eq:Hermite-ratio}. Thus, every coefficient is bounded above by the
coefficient obtained by merging two singletons. Note that
for $1\leqslant r\leqslant \ell-1\leqslant  n-4$ we have
\begin{equation*} C_r(2,2) = 2H_r-\left(r+\frac n2-1\right)H_{r-1}
\geqslant \left(2(n-1)-r-\frac n2+1\right)H_{r-1}
=\left(\frac{3n}{2}-r-1\right)H_{r-1}>0. 
\end{equation*}
Also $C_0(2,2)=2>0$. Since each $c_i\geqslant 1$,
$$E_j=e_j(c_1^{-1},\dots,c_\ell^{-1})\leqslant \binom\ell j.$$
For $r=\ell-j-1$, and $r\geqslant 1$ we have
\begin{equation*}
  C_r(2,2)
  =n^r\left[
      2h_r-
      \left(\frac{1}{2}+\frac{r-1}{n}\right)h_{r-1}
    \right]
  \leqslant n^r
    \left(\frac{3}{2}-\frac{r-1}{n}\right)h_r
  =n^r
    \left(\frac{3}{2}-\frac{\ell+2}{n}+\frac{j}{n}\right)h_r.
\end{equation*}
We have to be a bit careful in the case $r=0$. For $r=0$, the inequality $h_{r-1}\geqslant h_r$ used above is unavailable. Indeed, $C_0(2,2)=2$ while the above upper bound would give $\frac{3}{2}+\frac{1}{n}$. To compensate its contribution to $n\Delta$ we need to add an extra $\frac{\ell}{2n^{\ell-1}}$ term since $E_{\ell-1}\leq \ell$.  

Using
\eqref{eq:h-exponential}, we obtain
\begin{align}
 n\Delta
 \leqslant 
 \sum_{j=0}^{\ell-1}
 \binom\ell j\frac1{n^j}
 \exp\left(
   -\frac{(\ell-j-1)(\ell-j-2)}{2n^2}
 \right)
 \left(\frac{3}{2}-\frac{\ell-2}{n}+\frac{j}{n}\right)+\frac{\ell}{2n^{\ell-1}}.
 \label{eq:Delta-binomial}
\end{align}
Since $n\geqslant 10$ we have $\ell-1=n-t-3\geqslant \frac{n}{2}-2\geqslant 3$. Then it follows that $\frac{\ell}{2n^{\ell-1}}\leqslant \frac{3}{2n^2}$.
Set
\[
  r_0=\ell-1,
  \qquad
  A=\frac{r_0(r_0-1)}{2n^2},
  \qquad
  \beta=\frac{2r_0-1}{2n^2},
  \qquad
  x=\frac{\ee^\beta}{n},
\]
and
\[
  B=\frac{3}{2}-\frac{\ell-2}{n}.
\]
Since
\[
  (r_0-j)(r_0-j-1)
  \geqslant r_0(r_0-1)-(2r_0-1)j,
\]
we have
\begin{align*}
n\Delta&\leqslant  \sum_{j=0}^{\ell-1}
 \binom{\ell}{j}\frac{1}{n^j}
 \exp\left(
   -\frac{(\ell-j-1)(\ell-j-2)}{2n^2}
 \right)
 \left(\frac{3}{2}-\frac{\ell-2}{n}+\frac{j}{n}\right)+\frac{3}{2n^2}\\
 &\leqslant \sum_{j=0}^{\ell}
 \binom\ell j\frac{1}{n^j}
 \exp\left(
   -\frac{(\ell-j-1)(\ell-j-2)}{2n^2}
 \right)
 \left(\frac{3}{2}-\frac{\ell-2}{n}+\frac{j}{n}\right)+\frac{3}{2n^2}\\
 &\leqslant 
\sum_{j=0}^{\ell}
 \binom{\ell}{j}\frac{1}{n^j}
 \exp\left(
   -\frac{(\ell-1)(\ell-2)}{2n^2}+j\frac{2\ell-3}{2n^2}
 \right)
 \left(\frac{3}{2}-\frac{\ell-2}{n}+\frac{j}{n}\right)+\frac{3}{2n^2}\\
&=\sum_{j=0}^{\ell}
 \binom{\ell}{j}\frac{1}{n^j}\exp(-A+j\beta)\left(B+\frac{j}{n}\right)+\frac{3}{2n^2}\\
&=e^{-A}\sum_{j=0}^{\ell}
 \binom{\ell}{j}x^{j}\left(B+\frac{j}{n}\right)+\frac{3}{2n^2}\\
&=e^{-A}(1+x)^{\ell}\left(B+\frac{\ell x}{n(1+x)}\right)+\frac{3}{2n^2}.
\end{align*}
Hence
\begin{equation}\label{eq:Delta-compact}
  n\Delta
  \leq
  \ee^{-A}(1+x)^\ell
  \left(B+\frac{\ell x}{n(1+x)}\right)+\frac{3}{2n^2}
\end{equation}

Put $y=r_0/n$.  From $t+1\leqslant n/2$ we have
$y\geqslant \frac{1}{2}-\frac{2}{n}$. Moreover,
\[
  A=\frac{y^2}{2}-\frac{y}{2n},
  \qquad
  B=\frac{3}{2}-y+\frac{1}{n}.
\]
Since $\beta\leqslant 1/n$ and
$\ee^{1/n}\leqslant n/(n-1)$, we have $x\leqslant 1/(n-1)$.  Therefore
\begin{align*}
  -A+\ell\ln(1+x)\leqslant& -A+\ell x\\
  &\leqslant -\frac{y^2}{2}+\frac{y}{2n}+\frac{\ell}{n-1}\\
  &=-\frac{y^2}{2}+\frac{\ell-1}{n}+\frac{y}{2n}+\frac{\ell}{n-1}-\frac{\ell-1}{n}\\
  &=-\frac{y^2}{2}+y+\frac{y}{2n}+\frac{n+\ell-1}{n(n-1)}\\
  &\leqslant -\frac{y^2}{2}+y+\frac{y}{2n}+\frac{y}{n-1}+\frac{1}{n-1}\\
  &\leqslant -\frac{y^2}{2}+y+\frac{1}{2n}+\frac{1}{n-1}+\frac{1}{n-1}\\
   &\leqslant -\frac{y^2}{2}+y+\frac{3}{n}
\end{align*}
for $n\geqslant 5$.
Thus
\begin{equation} \label{eq:exponent-bound}
-A+\ell\ln(1+x)\leqslant -\frac{y^2}{2}+y+\frac{3}{n}
\end{equation}
Similarly,
\begin{equation}  \label{eq:factor-bound}
  B+\frac{\ell x}{n(1+x)}\leqslant B+\frac{\ell \frac{1}{n-1}}{n(1+\frac{1}{n-1})}=B+\frac{\ell}{n^2}\leqslant B+\frac{1}{n}=\frac{3}{2}-y+\frac{2}{n}.
\end{equation}
Hence
$$n\Delta\leqslant e^{y-y^2/2+3/n}\left(\frac{3}{2}-y+\frac{2}{n}\right)+\frac{3}{2n^2}$$
The function
$$\Phi_n(y)=
  \ee^{y-y^2/2+3/n}
  \left(\frac{3}{2}-y+\frac{2}{n}\right)
$$
is decreasing for $y\geqslant \frac{1}{2}-\frac{2}{n}$.  Indeed, its logarithmic derivative is
$$1-y-\frac{1}{3/2-y+2/n}\leqslant 0,$$
because
$$(1-y)\left(\frac{3}{2}-y+\frac{2}{n}\right)
  \leq
  \left(\frac{1}{2}+\frac{2}{n}\right)\left(1+\frac{4}{n}\right)
  \leqslant 1
$$
for $n\geqslant 10$.
It follows from \eqref{eq:Delta-compact}--\eqref{eq:factor-bound} that
\begin{equation}\label{eq:Delta-exp-bound}
  n\Delta
  \leq
  \ee^{\,3/8+2/n-2/n^2}
  \left(1+\frac{4}{n}\right)+\frac{3}{2n^2}.
\end{equation}
From \eqref{eq:delta-lower} we have
$$ 2n\delta_n\geqslant
  \frac{3}{2}+\frac{23n^2+34n-120}{2n^3}.$$
So it is enough to show that for $n\geqslant 10$ we have
$$\frac{3}{2}+\frac{23n^2+34n-120}{2n^3}>\ee^{\,3/8+2/n-2/n^2}\left(1+\frac{4}{n}\right)+\frac{3}{2n^2}.$$
This is indeed true as $e^{3/8}<\frac{3}{2}$ and $e^{2/n-2/n^2}<\left(1+\frac{1}{n-1}\right)^2$ and
\begin{equation} \label{eq:final-numerical-check}
\frac{3}{2}+\frac{23n^2+34n-120}{2n^3}-\frac{3}{2}\left(1+\frac{1}{n-1}\right)^2\left(1+\frac{4}{n}\right)-\frac{3}{2n^2}=\frac{5n^4 - 12n^3 - 159n^2 + 271n - 120}{2n^3(n-1)^2}
\end{equation}
After substituting $n=10+x$ into the numerator in \eqref{eq:final-numerical-check} it is equal to $5x^4+188x^3+2481x^2+13491x+24690$, so the numerator is positive for $n\geqslant 10$.  Combining
\eqref{eq:delta-lower} and
\eqref{eq:final-numerical-check} yields $n\Delta\leqslant 2n\delta_n$, proving
\eqref{eq:contraction-bound}.
\end{proof}

For a forest $S$ in $K_n$, let $\boldsymbol a(S)$ be the vector of the sizes
of its components, including isolated vertices, and define
\[
  d(S)=Q_n-Q(\boldsymbol a(S)).
\]
The preceding lemmas imply the following form that will be used in the final
counting argument.

\begin{Cor}\label{cor:defect-growth}
Let $n\geqslant 10$.  If $S\neq\varnothing$ is a forest with $t$ edges, $S+e$ is a
forest, and $t+1\leqslant n/2$, then
\begin{equation}\label{eq:defect-growth}
  d(S+e)\leqslant 3d(S).
\end{equation}
\end{Cor}

\begin{proof}
By Lemma~\ref{lem:contraction} and Corollary~\ref{cor:min-defect},
\[
  d(S+e)
  =d(S)+Q(\boldsymbol a(S))-Q(\boldsymbol a(S+e))
  \leqslant d(S)+2\delta_n
  \leq3d(S).
\]
\end{proof}

Next we give the proof of Lemma~\ref{Q_n-upper-bound}.

\begin{proof}[Proof of Lemma~\ref{Q_n-upper-bound}]
Write $F_n=F(K_n)$ as before. By Cayley's formula,
$T(K_n)=n^{n-2}$. The cases $n\leqslant  2$ are immediate, so assume $n\geq 3$.

Let $R(z)=\sum_{m\geq1}m^{m-1}\frac{z^m}{m!}$. It is well-known that $R=ze^R$, and that the exponential generating function of unrooted labelled trees is
$U(z)=R(z)-\frac{R(z)^2}{2}$, so the exponential generating function of labelled forests is $e^{U(z)}$.
Lagrange inversion gives
\begin{equation} \label{F_n-formula}
F_n
=n![z^n]e^{U(z)}
=(n-1)![u^{n-1}](1-u)e^{(n+1)u-u^2/2}.
\end{equation}
Define
$$L_j=j![u^j]e^{(n+1)u-u^2/2}.$$
Note that
$$\sum_{k\geqslant 0}\He_{k}(x)\frac{t^k}{k!}=e^{xt-t^2/2}.$$
Thus $L_j=\He_j(n+1)$ with our previous notation, hence it satisfies the following recursion:
$$
L_0=1,
\qquad L_1=n+1,
\qquad
L_j=(n+1)L_{j-1}-(j-1)L_{j-2},
$$
and \eqref{F_n-formula} yields
\begin{equation} \label{F_n-formula-2}
F_n=L_{n-1}-(n-1)L_{n-2}=2L_{n-2}-(n-2)L_{n-3}.
\end{equation}

Put $a=n+1$ and $r_j=L_j/L_{j-1}$. For $1\leq j\leq n-2$,
\[
r_j=a-\frac{j-1}{r_{j-1}},
\qquad n<r_j\leq a.
\]
Consequently,
\[
r_j
\leq a-\frac{j-1}{a-(j-2)/a}
=a\frac{a^2-2j+3}{a^2-j+2}.
\]
Using \eqref{F_n-formula-2} and the monotonicity of $x\mapsto 2-(n-2)/x$, we obtain
\begin{equation} \label{Q_n-bound-1}
\frac{F_n}{n^{n-2}}
\leq
\left(1+\frac{1}{n}\right)^{n-2}
\prod_{j=1}^{n-2}\frac{a^2-2j+3}{a^2-j+2}\,
\frac{n^3+3n^2+13n+26}{(n+1)(n^2+8)}.
\end{equation}

For $j\geqslant 2$, set
\[
x_j=\frac{2j-3}{a^2},
\qquad
y_j=\frac{j-2}{a^2}.
\]
Since $0\leq y_j\leq x_j<1$,
\[
\ln\left(\frac{1-x_j}{1-y_j}\right)=-\int_{y_j}^{x_j}\frac{1}{1-t}\ dt\leqslant -\int_{y_j}^{x_j}(1+t+t^2)\ dt
=
-\sum_{s=1}^{3}\frac{x_j^s-y_j^s}{s}.
\]
The three elementary power sums therefore turn \eqref{Q_n-bound-1} into
\begin{align}
\ln\left(\frac{F_n}{n^{n-2}}\right)
\leqslant&
(n-2)\ln\left(1+\frac{1}{n}\right)
+\ln\frac{n^3+3n^2+13n+26}{(n+1)(n^2+8)} \\
&-\frac{(n-3)(n-2)}{2(n+1)^2}
-\frac{(n-3)(n-2)(2n-7)}{4(n+1)^4}
\notag\\
&-\frac{(n-3)^2(n-2)(7n-26)}{12(n+1)^6}.
\end{align}

From Lemma~\ref{Lemma-ln(1+x)} we use the bounds
$$\ln\left(1+\frac{1}{n}\right)\leqslant \frac{1}{n}-\frac{1}{2n^2}+\frac{1}{3n^3}$$
and the bound
$$\ln(1+s)\leqslant \frac{s(2+s)}{2(1+s)},$$
where $s=\frac{2n^2+5n+18}{(n+1)(n^2+8)}<1$ for $n>2$. 
This gives the following upper bound for $\ln(Q_n)$;
\begin{align*}
\ln(Q_n)\leqslant &\frac{6n^{14} + 84n^{13} + 576n^{12} + 2765n^{11} + 9461n^{10} + 24722n^9 + 40743n^8}{12(n^3 + 3n^2 + 13n + 26)(n^2 + 8)(n + 1)^6n^3}\\
&
\quad +\frac{50723n^7 - 5236n^6- 83164n^5 + 71916n^4 - 100640n^3 - 14960n^2 - 7488n - 1664}{12(n^3 + 3n^2 + 13n + 26)(n^2 + 8)(n + 1)^6n^3}
\end{align*}

This is less than $\frac{1}{2}+\frac{5}{2n}-\frac{3}{2n^2}$ because

\begin{align*}&\frac{1}{2}+\frac{5}{2n}-\frac{3}{2n^2}-RHS\\
&=\frac{139n^{11} + 1561n^{10} + 7078n^9 + 31407n^8 + 73057n^7 + 149296n^6}{12(n^3 + 3n^2 + 13n + 26)(n^2 + 8)(n + 1)^6n^3}\\
&\quad \quad +\frac{n^4(177736n - 58722) + n^2(74156n - 3136) + 3744n + 1664}{12(n^3 + 3n^2 + 13n + 26)(n^2 + 8)(n + 1)^6n^3}
\end{align*}
which is clearly positive for $n\geqslant 1$.
Hence 
$$Q_n<\exp\left(\frac{1}{2}+\frac{5}{2n}-\frac{3}{2n^2}\right).$$
\end{proof}

Next we consider the inequality $D_1\geq D_2$. This requires a bound $d(S+e)\leq \alpha d(S)$ for $|S|=1$ with sufficiently small $\alpha$.

\begin{Lemma}[Two-edge estimate]\label{lem:two-edge}
Let $R$ be any fixed two-edge forest in $K_n$.  Then for any $\alpha>2$ there exists an $n(\alpha)$ such that for $n\geqslant n(\alpha)$ we have
\begin{equation}\label{eq:two-edge-defect}
  Q_n-Q(\boldsymbol a(R))\leqslant \alpha\delta_n.
\end{equation}
One can choose $n(\alpha)=\lceil\frac{5\alpha-7}{\alpha-2}\rceil$, and $\alpha=2.04$ works for all $n\geqslant 4$. 
\end{Lemma}

\begin{proof}
Since the function $Q(\boldsymbol{a})$ is Schur-convex we have
$$\max_{|R|=2}(Q_n-Q(\boldsymbol a(R))=\max\left(Q_n-Q(3,1^{n-3}),Q_n-Q(2,2,1^{n-4})\right)=Q_n-Q(2,2,1^{n-4}).$$
This means that we only need  to consider the case of two vertex-disjoint edges. So, suppose that the two edges of $R$ are disjoint. 
Let $\cF_{n,k}$ be a uniformly random $k$-component spanning forest of $K_n$,
and put
\[
  r_k(R)=
  \frac{\Pr(R\subseteq\cF_{n,k})}
       {\Pr(R\subseteq\cF_{n,1})}.
\]
We shall prove that
\begin{equation}\label{eq:rk-bound}
  r_k(R)\geqslant 1-\alpha \frac{(k-1)}{n-1}.
\end{equation}
if $n\geqslant \frac{5\alpha-7}{\alpha-2}$. This is trivial for $k=1$, so we assume $k\geqslant 2$ from now on.
\medskip

For a forest $F$, let
\[
  A(F)=\sum_v\binom{\deg_F(v)}2,
\]
the number of unordered adjacent pairs of edges.  
 The total number of
unordered pairs of edges in a $k$-component forest is
$\binom{n-k}{2}$.  Thus the number of disjoint pairs is
\[
  M(F)=\binom{n-k}{2}-A(F).
\]
Note that
$$r_k(R)=\frac{\mathbb{E}M}{\mathbb{E}M_{\mathrm{tree}}}.$$
The number of labelled trees on $s$ vertices with degree sequence $d_1,\dots, d_s$ satisfying $d_1+\dots+d_s=2(s-1)$ and $d_i\geqslant 1$ is
$$\binom{s-2}{d_1-1,\dots ,d_s-1}.$$
Since 
$$\sum_{d_1,\dots ,d_s}\binom{s-2}{d_1-1,\dots ,d_s-1}x_1^{d_1-1}\dots x_s^{d_s-1}=(x_1+\dots +x_s)^{s-2}$$
Differentiating once and twice with respect to $x_1$, respectively, gives 
$$\sum_{d_1,\dots ,d_s}\binom{s-2}{d_1-1,\dots ,d_s-1}(d_1-1)x_1^{d_1-2}\dots x_s^{d_s-1}=(s-2)(x_1+\dots +x_s)^{s-3}$$
$$\sum_{d_1,\dots ,d_s}\binom{s-2}{d_1-1,\dots ,d_s-1}(d_1-1)(d_1-2)x_1^{d_1-3}\dots x_s^{d_s-1}=(s-2)(s-3)(x_1+\dots +x_s)^{s-4}.$$
Plugging $x_1=\dots =x_s=1$ and using the symmetry of the vertices we get that for a uniform tree on $s$ vertices we have
\[
  \mathbb E\left[\sum_v\binom{\deg(v)}2\right]
  =\frac{3(s-1)(s-2)}{2s}
\]
For $s=1$, the same formula holds trivially.
By conditioning on the component
sizes $s_1,\dots,s_k$ of a $k$-forest we get that
\[
  \mathbb E(A\mid s_1,\dots,s_k)
  =\frac{3}{2}\left(n-3k+2\sum_{i=1}^k\frac{1}{s_i}\right).
\]
For positive integers $s_i$ with sum $n$,
\[
  \sum_{i=1}^k\frac{1}{s_i}
  \leqslant k-1+\frac{1}{n-k+1}.
\]
Consequently,
\[
  \mathbb EA\leq
  \frac{3}{2}\left(n-k-2+\frac{2}{n-k+1}\right).
\]
Then
\begin{align*}
&\mathbb{E}M-\left(1-\alpha\frac{(k-1)}{n-1}\right)\mathbb{E}M_{\mathrm{tree}}\\
&\quad \geqslant \binom{n-k}{2}-\frac{3}{2}\left(n-k-2+\frac{2}{n-k+1}\right)-\left(1-\alpha\frac{(k-1)}{n-1}\right)\left(\binom{n-1}{2}-\frac{3}{2}\left(n-3+\frac{2}{n}\right)\right)\\
&\quad =(k-1)\left(\left(\frac{\alpha}{2}-1\right)n+\left(3-\frac{5}{2}\alpha\right)+\frac{k-1}{2}-\frac{3}{n(n-k+1)}+\frac{3\alpha}{n}\right)
\end{align*}
Here $\frac{3\alpha}{n}>\frac{6}{n}>\frac{3}{n(n-k+1)}$. We also have $k\geqslant 2$ implying $\frac{k-1}{2}\geqslant \frac{1}{2}$, thus
$$\mathbb{E}M-\left(1-\alpha\frac{(k-1)}{n-1}\right)\mathbb{E}M_{\mathrm{tree}}\geqslant (k-1)\left(\left(\frac{\alpha}{2}-1\right)n+\left(\frac{7}{2}-\frac{5}{2}\alpha\right)\right)$$
which is non-negative if $n\geqslant n(\alpha)=\lceil\frac{5\alpha-7}{\alpha-2}\rceil$.

Finally,
\[
  Q(\boldsymbol a(R))
  =\sum_{k=1}^n\frac{f_{n,k}}{T_n}r_k(R).
\]
Therefore, by \eqref{eq:rk-bound} and \eqref{eq:delta-sum},
\[
  Q_n-Q(\boldsymbol a(R))
  \leqslant
  \alpha \sum_{k=2}^n\frac{f_{n,k}}{T_n}\frac{k-1}{n-1}
  = \alpha\delta_n.
\]
To prove the claim about $\alpha=2.04$, first observe that 
$n(2.04)=\frac{5\cdot 2.04-7}{2.04-2}=80$, thus for $n\geqslant 80$ the previous statement confirms the claim.
It remains to verify $4\leqslant n\leqslant 80$. Exact rational computation gives
that $\frac{Q_n-Q(2,2,1^{n-4})}{\delta_n}\leqslant 2.04$ for $4\leqslant n\leqslant 80$. In fact, we have $n(2.03)=105$ and
\begin{equation} \label{eq:Q-ineq-2}
\max_{4\leqslant n\leqslant 105}\frac{Q_n-Q(2,2,1^{n-4})}{\delta_n}=\frac{Q_{19}-Q(2,2,1^{15})}{\delta_{19}}\approx 2.03660399434459...
\end{equation}
which shows that this is the maximum for all $n$.  This completes the proof.
\end{proof}

\begin{Rem}

For completeness, we share the computational results for $n\leqslant 20$. Further data is available in the Jupyter notebook.
\bigskip

\begin{center}
\begin{tabular}{|c|c|c|}  \hline
$n$ & $(Q_n-Q(2,2,1^{n-4}))/\delta_n$ & $(Q_n-Q(3,1^{n-3}))/\delta_n$ \\ \hline
4 & 1.80000000000000 & 1.66666666666667 \\ \hline
5 &1.87606837606838 &1.69800569800570 \\ \hline
6 &1.92703862660944 &1.70815450643777 \\ \hline
7 &1.96139379375196 &1.70626545467210 \\ \hline
8 &1.98479763079961 &1.69786113853241 \\ \hline
9 &2.00092566817026 &1.68615930141107 \\ \hline
10 &2.01215245842404 &1.67300132181566 \\ \hline
11 &2.02002434873254 &1.65942696120447 \\ \hline
12 &2.02556253347287 &1.64601256025117 \\ \hline
13 &2.02945257898197 &1.63306657477311 \\ \hline
14 &2.03216260579231 &1.62074242848301 \\ \hline
15 &2.03401756911918 & 1.60910399613953 \\ \hline
16 &2.03524656688534 &1.59816389729430 \\ \hline
17 &2.03601344258242 & 1.58790605540397 \\ \hline
18 &2.03643690838386 & 1.57829905341418 \\ \hline
19 &2.03660399434459 &1.56930404644049 \\ \hline
20 &2.03657917958015 & 1.56087942331356 \\ \hline
\end{tabular}
\end{center}

\end{Rem}

Now we are ready to complete the argument in the introduction of this section.

\begin{proof}[Proof of Theorems~\ref{forest-tree-ratio-3} and  \ref{forest-tree-ratio-3-quantitative}]
Corollary~\ref{cor:min-defect} gives that $d(S)\geqslant 0$ for every $S$. For $n\geqslant 10$ Corollary~\ref{cor:defect-growth} shows that if $S\neq\varnothing$ is a forest with $t$ edges, $S+e$ is a
forest, and $t+1\leqslant \frac{n}{2}$, then we have
\begin{equation}\label{eq:defect-growth}
  d(S+e)\leqslant 3d(S).
\end{equation}
 By the inequality \eqref{eq:Ds-ratio} we have
$\frac{D_2}{D_1}\leqslant \frac{2\cdot 2.04(h-1)}{2n}<1$
if $h \leqslant \frac{1}{2.04}n+1$. Note that $\frac{1}{2.04}>0.49$. By \eqref{eq:Ds-ratio} we also have $\frac{D_2}{D_1}\leqslant \frac{2\alpha(h-1)}{2n}$ if $\alpha>2$ and $n\geqslant n(\alpha)$. By choosing $2<\alpha<\frac{1}{1/2-\varepsilon}$ this implies that for sufficiently large $n$ we have $D_1>D_2$ whenever $h\leqslant \left(\frac{1}{2}-\varepsilon\right)n$. This immediately completes the proof of Theorem~\ref{forest-tree-ratio-3}. 
Theorem~\ref{forest-tree-ratio-3-quantitative} also follows immediately from Lemma~\ref{lem:two-edge} and Corollary~\ref{cor:defect-growth}.
\end{proof}

\begin{proof}[Proof of Theorem~\ref{forest-tree-ratio-thm}]
For $n\geqslant 10$ the theorem immediately follows from  Theorems~\ref{forest-tree-ratio-2-quantitative} and \ref{forest-tree-ratio-3-quantitative}. As before let $h=\binom{n}{2}-m$, where $m$ is the number of edges of $G$. For $n=8, 9$ observe that  Theorem~\ref{forest-tree-ratio-2-quantitative} shows that $\lfloor 0.49n +1/2\rfloor = 4$, so if $h\geqslant 4$, then we are done. If $0<h\leqslant 3$, then we only need to check the inequality $D_1>D_2$ since $D_3\geqslant 0$.  By the inequality \eqref{eq:Ds-ratio}  we indeed have $\frac{D_2}{D_1}\leqslant \frac{2\cdot 2.04\cdot 2}{2\cdot 8}<1$. For $n\leqslant 7$ one can check the statement by computer, see the Jupyter notebook. (For $n=7$ one can also use the same argument as for $n=8,9$, but instead of $\lfloor 0.49n+1/2\rfloor$ we need to use the value $h_7=4$ from Table 1.) 
\end{proof}

\section{Miscellaneous remarks}

In this section, we collect various remarks.

\subsection{How natural is Theorem~\ref{ratio-minimum}?}

Let $P(x)=\sum_{k=0}^ra_kx^k$ and $Q(x)=\sum_{j=0}^sb_jx^j$ be polynomials with non-negative coefficients, and put $P(1)=A$ and $Q(1)=B$. If $R(x)=P(x)Q(x)=\sum_{t=0}^{r+s}c_tx^t$ and $R(1)=C$, then we have
$$\ln\left(\frac{C}{c_{r+s}}\right)=\ln\left(\frac{A}{a_r}\right)+\ln\left(\frac{B}{b_s}\right)\ \ \ \ \text{and}\ \ \ \ \frac{c_{r+s-1}}{c_{r+s}}=\frac{a_{r-1}b_{s}+a_rb_{s-1}}{a_rb_s}=\frac{a_{r-1}}{a_r}+\frac{b_{s-1}}{b_s},$$
so both expressions are additive. Furthermore, if $P(x)=\prod_{k=1}^r(1+\lambda_kx)$, where $\lambda_1,\dots , \lambda_r$ are positive real numbers, then
$$\ln \left(\frac{A}{a_r}\right)=\sum_{k=1}^r\ln\left(1+\frac{1}{\lambda_k}\right)\leqslant \sum_{k=1}^r\frac{1}{\lambda_k}=\frac{a_{r-1}}{a_r}.$$
In fact, we can improve this inequality by observing that the function $x\ln\left(1+\frac{1}{x}\right)$ is increasing. If $\lambda_r=\max_{1\leqslant k\leqslant r}\lambda_k$, then
$$\ln\left(\frac{A}{a_r}\right)=\sum_{k=1}^r\ln\left(1+\frac{1}{\lambda_k}\right)\leqslant  \sum_{k=1}^r\frac{\lambda_r}{\lambda_k}\ln\left(1+\frac{1}{\lambda_r}\right)=\frac{a_{r-1}}{a_r}\cdot \lambda_r\ln\left(1+\frac{1}{\lambda_r}\right).$$
Thus, for real-rooted $P(x)$ we have, in fact, an opposite inequality. A notable example is the matching generating polynomial. Let $m_k(G)$ be the number of matchings of size $k$ of a graph $G$. Then the polynomial $\sum_{k=0}^rm_k(G)x^k$ is real-rooted, and all $\lambda_i$ are at most $4(\Delta-1)$ if $\Delta\geqslant 2$, where $\Delta$ is the maximum degree of the graph, see for instance \cite{heilmann1972theory}.

Theorem~\ref{ratio-minimum} is approximately tight when there is a hidden Poisson distribution in the following sense. If $\frac{a_{r-1}}{a_r}=\alpha$ and $\frac{a_{r-j}}{a_r}\sim \frac{\alpha^j}{j!}$, then, under suitable tail assumptions, we have
$$\frac{A}{a_r}\sim \sum_{j=0}^r \frac{\alpha^j}{j!}\sim e^{\alpha}.$$
This is exactly the case with the complete graph $K_n$, where $\frac{F_k(K_n)}{T(K_n)}\sim \frac{1}{2^{k-1}(k-1)!}$.
This cannot work for general matroids, because if $M=(E,\mathcal{I})$ is a matroid of rank $r$, then $M'=(E,\mathcal{I}')$ with $\mathcal{I}'=\{S\in \mathcal{I}\ |\ |S|\leqslant r-1\}$ is also a matroid, but then if $I_{r-1}=\alpha I_r$, $I_{r-2}\sim \frac{\alpha^2}{2}I_r$ and $I_{r-3}\sim \frac{\alpha^3}{6}I_r$, then
$$\frac{I_{r-2}}{I_{r-1}}\sim \frac{\alpha}{2},\ \ \text{but}\ \ \ \frac{I_{r-3}}{I_{r-1}}\sim \frac{\alpha^2}{6}\neq \frac{(\alpha/2)^2}{2},$$
so the hidden Poisson distribution is not inherited. Nevertheless, there is some intuition for why Theorem~\ref{ratio-minimum} works quite well in the case of dense graphs. First, most spanning trees of a dense graph have similar shapes, so most local basis exchange graphs $H[A]$ look alike. Second, the bound $\prod_{e\in A}\left(1+\frac{1}{\delta(A-e)}\right)$ is not bad for $\widetilde{T}_{H[A]}(2,1)$ if the neighbors of $e_1, e_2\in A$ are almost disjoint whenever the sizes of $N_{H[A]}(e_1)$ and $N_{H[A]}(e_2)$ are small. In the case of spanning trees, this neighbor set is typically small if the edge $e$ is incident to a leaf vertex of the spanning tree, but if $e_1$ and $e_2$ are both incident to leaf vertices, then there can be at most one non-tree edge $f$ that is adjacent to both $e_1, e_2$ in $H[A]$, namely the edge connecting the leaves (which may not even be an edge of the graph).  

\subsection{The function $Q(\boldsymbol{a})$}
One can prove that if $\boldsymbol{a}=(a_1,\dots ,a_m) \in \mathbb{Z}_{>0}^m$ is such that $a_1+\dots +a_m=n$, then 
$$Q_n-Q(\boldsymbol{a})=\frac{\sqrt{e}}{n}\sum_{i=1}^m\left(1-\frac{1}{a_i}\right)+O\left(\frac{(n-m)^2}{n^2}\right).$$
In fact, the following explicit error term estimate can be given.

\begin{Th} Let $\boldsymbol{a}=(a_1,\dots ,a_m) \in \mathbb{Z}_{>0}^m$ such that $a_1+\dots +a_m=n$.
For $n\geqslant 20$ and $m\geqslant \frac{n}{2}$ we have
$$\left|Q_n-Q(\boldsymbol{a})-\frac{\sqrt{e}}{n}\sum_{i=1}^m\left(1-\frac{1}{a_i}\right)\right|\leqslant 250\frac{(n-m)^2}{n^2}.$$
\end{Th}

We omit the proof of this theorem. This result explains quite a lot of the phenomena we have seen previously. For example,
$$\delta_n=Q_n-Q(2,1^{n-2})=\frac{\sqrt{e}}{2n}+O\left(\frac{1}{n^2}\right)$$
and that
$$\frac{Q_n-Q(2,2,1^{n-4})}{\delta_n}=2+O\left(\frac{1}{n}\right).$$
In the latter case, we only gave an upper bound $2+O\left(\frac{1}{n}\right)$ previously, so this is actually a stronger statement.

\subsection{RSW-formula for general matroids}

The RSW formula implies that, for graphic matroids,
we have
$$\frac{I_{r-1}}{I_r}\geqslant \frac{1}{4}\sum_{e\in E}R_{\eff}(e^-,e^+)^2=\frac{1}{4}\sum_{e\in E}\mathbb{P}(e\in \textbf{B})^2,$$
where $\textbf{B}$ is a uniformly random basis (spanning tree) of the graphic matroid. This inequality seems to be rather special for graphic matroids. In general, we can only claim much weaker statements.

\begin{Prop} \label{MSW-general-matroids}
Let $M=(E,\mathcal{I})$ be a matroid of rank $r\geqslant 1$. Let $\textbf{B}$ be a uniformly random basis. Then
$$\frac{1}{r}\sum_{e\in E}\mathbb{P}(e\in \textbf{B})^2\leqslant \frac{I_{r-1}}{I_r}\leqslant \sum_{e\in E}\mathbb{P}(e\in \textbf{B})^2.$$
\end{Prop}

\begin{proof}
For an $I\in \mathcal{I}_{r-1}$ let
$$C(I)=\{e\in E\ |\ I\cup \{e\}\in \mathcal{I}_r\}.$$
Observe that every basis $B$ contains at least one element of $C(I)$, otherwise the augmentation axiom would fail for $I$ and $B$. Hence
$$1\leqslant \sum_{e\in C(I)}\mathbb{P}(e\in \textbf{B})=\mathbb{E}[|C(I)\cap \textbf{B}|]\leqslant r.$$
Adding up these inequalities for all $I\in \mathcal{I}_{r-1}$ we get that
$$I_{r-1}\leqslant \sum_{I\in \mathcal{I}_{r-1}}\sum_{e\in C(I)}\mathbb{P}(e\in \textbf{B})\leqslant rI_{r-1}.$$
Now observe that
$$\sum_{I\in \mathcal{I}_{r-1}}\sum_{e\in C(I)}\mathbb{P}(e\in \textbf{B})=\sum_{e\in E
}(I_r\cdot \mathbb{P}(e\in \textbf{B}))\cdot \mathbb{P}(e\in \textbf{B})=I_r\sum_{e\in E}\mathbb{P}(e\in \textbf{B})^2$$
as every $e\in E$ appears as an element of $C(I)$ for exactly $I_r\cdot \mathbb{P}(e\in \textbf{B})$ elements of $\mathcal{I}_{r-1}$. Hence 
$$\frac{1}{r}\sum_{e\in E}\mathbb{P}(e\in \textbf{B})^2\leqslant \frac{I_{r-1}}{I_r}\leqslant \sum_{e\in E}\mathbb{P}(e\in \textbf{B})^2.$$
\end{proof}

\begin{Rem}
While for graphic matroids we have
$$\frac{I_{r-1}}{I_r}\geqslant \frac{1}{4}\sum_{e\in E}\mathbb{P}(e\in \textbf{B})^2,$$
for the uniform matroid $U_{r,m}$ we have
$$\frac{I_{r-1}}{I_r}=\frac{m}{r(m-r+1)}\sum_{e\in E}\mathbb{P}(e\in \textbf{B})^2$$
showing that $\frac{1}{r}$ cannot be improved further without knowing $m$. In fact, the constant $\frac{m}{r(m-r+1)}$ gives the best possible lower bound: or every rank-$r$ matroid $M$ on $m$ elements we have
$$\frac{I_{r-1}}{I_r}\geqslant \frac{m}{r(m-r+1)}\sum_{e\in E}\mathbb{P}(e\in \textbf{B})^2.$$
This strengthening of the lower bound in Proposition~\ref{MSW-general-matroids} was proved by ChatGPT 5.5 Pro. See Proposition~\ref{thm:basic-marginal-bound} below. 
One might hope that at least for regular matroids there is a universal positive constant $c$ such that
$$\frac{I_{r-1}}{I_r}\geqslant c\sum_{e\in E}\mathbb{P}(e\in \textbf{B})^2.$$
However, this is not the case: the cographic matroids of complete graphs show that there is no such universal constant $c$ for regular matroids. For more details, see Proposition~\ref{prop:cographic-counterexample} below.
\end{Rem}

\begin{Prop}[ChatGPT 5.5 Pro]\label{thm:basic-marginal-bound}
Let $M=(E,\cI)$ be a matroid of rank $r\geq 1$ on $m$ elements. Let $\textbf{B}$ be a uniformly random basis. Then
\[
 \frac{I_{r-1}}{I_r}
 \geqslant
 \frac{m}{r(m-r+1)}
 \sum_{e\in E}\PP(e\in \textbf{B})^2.
\]
\end{Prop}

\begin{proof} Let $\cB(M)$ denote the set of bases of $M$. For $e\in E$, put
$$b_e:=|\{B'\in\cB(M)\ |\ e\in B'\}|, \qquad p_e:=\PP(e\in \textbf{B})=\frac{b_e}{I_r}. $$
Then $\sum_{e\in E} b_e=rI_r$.
We first prove the pointwise estimate
\begin{equation} \label{eq: pointwise}
 b_e\leq \frac{m-r+1}{m}\,I_{r-1}
 \qquad\text{for every }e\in E. 
\end{equation}
If $e$ is a loop, then $b_e=0$, so there is nothing to prove. Assume that $e$ is not a loop. The case $r=1$ is immediate: $I_0=1$, and every non-loop element is itself a basis, so $b_e=1$ and equality holds in \eqref{eq: pointwise}.

Now assume $r\geq 2$ and let $N=M/e$. Then $N$ has rank $r-1$ on $m-1$ elements. Bases of $N$ correspond exactly to bases of $M$ containing $e$, after deleting $e$. Hence
\[
 b_e=I_{r-1}(N).
\]
Also, independent $(r-1)$-sets of $N$ are independent $(r-1)$-sets of $M$ not containing $e$, and independent $(r-2)$-sets of $N$, after adjoining $e$, give independent $(r-1)$-sets of $M$ containing $e$. Therefore
\begin{equation} \label{eq: decomposition}
 I_{r-1}(M)\geqslant I_{r-1}(N)+I_{r-2}(N). 
\end{equation}

We next compare $I_{r-2}(N)$ and $I_{r-1}(N)$ by a simple double count. Count pairs $(A,C)$, where $C$ is a basis of $N$ and $A\subset C$ has size $r-2$. Each basis $C$ contains $r-1$ such subsets, so the number of pairs is
\[
 (r-1)I_{r-1}(N).
\]
On the other hand, a fixed independent $(r-2)$-set $A$ can be extended to a basis of $N$ by adding at most
\[
 (m-1)-(r-2)=m-r+1
\]
elements. Thus
\[
 (r-1)I_{r-1}(N)
 \leqslant
 (m-r+1)I_{r-2}(N),
\]
and hence
\[
 I_{r-2}(N)
 \geqslant
 \frac{r-1}{m-r+1} I_{r-1}(N).
\]
Combining this with \eqref{eq: decomposition}, we obtain
\[
 I_{r-1}(M)
 \geqslant
 \left(1+\frac{r-1}{m-r+1}\right)I_{r-1}(N)
 =
 \frac{m}{m-r+1} I_{r-1}(N)
 =
 \frac{m}{m-r+1}b_e.
\]
This proves \eqref{eq: pointwise}.

Now use \eqref{eq: pointwise} as follows:
\[
 \sum_{e\in E} b_e^2
 \leqslant
 \frac{m-r+1}{m} I_{r-1}\sum_{e\in E} b_e
 =
 \frac{r(m-r+1)}{m}I_{r-1}I_r.
\]
Dividing by $I_r^2$ gives
\[
 \sum_{e\in E}p_e^2
 \leq
 \frac{r(m-r+1)}{m}\frac{I_{r-1}}{I_r},
\]
which is equivalent to the desired inequality.
\end{proof}

\begin{Prop}[ChatGPT 5.5 Pro]\label{prop:cographic-counterexample}
Let $n\geqslant 3$. Let $M=M^*(K_n)$ be the cographic matroid of the complete graph $K_n$. Let $\textbf{B}$ be a uniformly random basis of $M$. Then $M$ is regular and
\[
 \frac{I_{r-1}(M)}{I_r(M)}
 =
 \frac{(n-1)!}{2n^{n-2}}
 \sum_{j=0}^{n-3}\frac{n^j}{j!},
\]
whereas
\[
 \sum_{e\in E(K_n)}\PP(e\in \textbf{B})^2
 =
 \binom {n}{2}\left(\frac{n-2}{n}\right)^2
 =
 \frac{(n-1)(n-2)^2}{2n}.
\]
Consequently, there is no positive $c$ such that
$$\frac{I_{r-1}(M)}{I_r(M)}\geqslant c\sum_{e\in E(K_n)}\PP(e\in \textbf{B})^2$$
holds true for all $n$.
\end{Prop}

\begin{proof}
The matroid $M^*(K_n)$ is regular because graphic matroids are regular and regularity is closed under duality. The ground set is $E(K_n)$, and
\[
 |E(K_n)|=\binom n2,
 \qquad
 r=\binom {n}{2}-(n-1).
\]
Bases of $M^*(K_n)$ are complements of spanning trees of $K_n$. Therefore
\[
 I_r(M)=T(K_n)=n^{n-2},
\]
by Cayley's formula.

Let $B$ be a uniformly random basis of $M^*(K_n)$. Then $B=E(K_n)\setminus T$, where $T$ is a uniformly random spanning tree of $K_n$. By symmetry,
\[
 \PP(e\in T)=\frac{n-1}{\binom{n}{2}}=\frac{2}{n}.
\]
Hence
\[
 \PP(e\in \textbf{B})=1-\frac{2}{n}=\frac{n-2}{n},
\]
and therefore
\[
 \sum_{e\in E(K_n)}\PP(e\in B)^2
 =
 \binom{n}{2}\left(\frac{n-2}{n}\right)^2.
\]

It remains to compute $I_{r-1}(M)$. A set $A\subseteq E(K_n)$ of size $r-1$ is independent in the dual matroid $M^*(K_n)$ if and only if its complement contains a spanning tree of $K_n$. Since
\[
 |E(K_n)\setminus A|=\binom n2-(r-1)=n,
\]
this complement is a connected labelled graph on $n$ vertices with $n$ edges, equivalently a connected unicyclic labelled graph. Thus $I_{r-1}(M)$ is the number $U_n$ of connected unicyclic labelled graphs on $n$ vertices.

We count these graphs according to the length $k$ of the unique cycle. First choose the $k$ cycle vertices, then arrange them in a cycle, and finally attach rooted forests to the cycle vertices. The number of rooted forests on $[n]$ rooted at a prescribed set of $k$ roots is $k n^{n-k-1}$. Consequently
\[
 U_n
 =
 \sum_{k=3}^n
 \binom {n}{k}\frac{(k-1)!}{2}\, k n^{n-k-1}
 =
 \frac{(n-1)!}{2}
 \sum_{j=0}^{n-3}\frac{n^j}{j!},
\]
where $j=n-k$. This proves the displayed formula for $I_{r-1}(M)/I_r(M)$.

Observe that
$$\binom{n}{2}\left(\frac{n-2}{n}\right)^2\sim \frac{n^2}{2}.$$
If $X_n$ is a Poisson random variable of mean $n$, then $\P(X_n\leqslant n-3)\to \frac{1}{2}$ by the central limit theorem, and we have
\begin{align*}
\frac{U_n}{T_n}&=\frac{1}{n^{n-2}}\frac{(n-1)!}{2}
 \sum_{j=0}^{n-3}\frac{n^j}{j!}\\
 &=\frac{e^n(n-1)!/2}{n^{n-2}}\cdot e^{-n}\sum_{j=0}^{n-3}\frac{n^j}{j!}\\
 &=\frac{e^n(n-1)!}{2n^{n-2}}\cdot \PP(X_n\leqslant n-3)\\
 &\sim \frac{e^n \sqrt{2\pi (n-1)} \left(\frac{n-1}{e}\right)^{n-1}}{2n^{n-2}}\cdot \frac{1}{2}\\
 &\sim \frac{e\sqrt{2\pi}}{4}(n-1)^{3/2}\left(\frac{n-1}{n}\right)^{n-2}\\
 &\sim \sqrt{\frac{\pi}{8}}n^{3/2}.
\end{align*}
Hence
$$\frac{U_n/T_n}{\sum_{e\in E(K_n)}\PP(e\in \textbf{B})^2}\sim \sqrt{\frac{\pi}{2n}}$$
showing that this ratio can be arbitrarily small.

\end{proof}

\subsection{Non-monotonicity}

A natural idea to prove that
$Q(G)\geqslant Q(K_n)$
for every connected graph $G$ is to show that for any edge $e\notin E(G)$ we have
$$Q(G)\geqslant Q(G+e).$$
However, this inequality fails in certain cases. For instance, let $G=K_{2,m}$, where $m\geqslant 7$ and let $e$ be the edge joining the vertices of degree $m$. Indeed, let $B_m=K_{2,m}+(u,v)$, where $u$ and $v$ are the vertices of degree $m$ in $K_{2,m}$. Then
$$T(K_{2,m})=m2^{m-1},\ \  T(B_m)=(m+2)2^{m-1}   ,\ \ F(K_{2,m})=(m+3)3^{m-1},\ \ F(B_m)=(m+6)3^{m-1}$$
and so
$$F(B_m)T(K_{2,m})-F(K_{2,m})T(B_m)=(m-6)6^{m-1}$$
which is positive if $m\geqslant 7$.

However, the monotonicity might nevertheless hold once $G$ is sufficiently dense with a sufficiently large minimum cut. Then, combined with the tools in Section~\ref{forest-tree-ratio-2-sect} one may give a simpler proof that the complete graph $K_n$ minimizes $Q(G)$ among simple connected graphs on $n$ vertices. 

\subsection{Dense graph limits}

In a forthcoming paper, we plan to publish the following result. 

\begin{Th}
Let $(G_n)_n$ be a sequence of dense graphs converging to a graphon $W$ whose minimum cut sizes $\delta(G_n)$ satisfy 
$$\liminf_{n\to \infty}\frac{\delta(G_n)}{v(G_n)}>0.$$
Let
$$d_W(x)=\int_0^1 W(x,y)\ dy$$
and
$$
\lambda_W=
\int_0^1\frac{dx}{d_W(x)}
-
\frac{1}{2}\int_{[0,1]^2}\frac{W(x,y)}{d_W(x)d_W(y)}\,dx\,dy.
$$
Then 
$$\lim_{n\to \infty}\frac{F_2(G_n)}{T(G_n)}=\lambda_W.$$
Furthermore, for every fixed $k$ we have
$$\lim_{n\to \infty}\frac{F_k(G_n)}{T(G_n)}=\frac{\lambda_W^{k-1}}{(k-1)!}$$
and
$$\lim_{n\to \infty}\frac{F(G_n)}{T(G_n)}=e^{\lambda_W}.$$
\end{Th}

Clearly, this statement is motivated by the fact that for $G_n=K_n$  we have $\lambda_J=\frac{1}{2}$ for the graphon $J(x,y)\equiv 1$. 

Another notable case concerns complete bipartite graphs $K_{n,\lfloor\beta n\rfloor}$ with a fixed $\beta>0$. It follows from the RSW-formula  that for the sequence $K_{n,\lfloor\beta n\rfloor}$ we have
$$\lim_{n\to \infty}\frac{F_2(K_{n,\lfloor\beta n\rfloor})}{T(K_{n,\lfloor\beta n\rfloor})}=\beta+\frac{1}{\beta}-1$$
and Stark \cite{stark2013asymptotic} proved that, indeed,
$$\lim_{n\to \infty}\frac{F(K_{n,\lfloor\beta n\rfloor})}{T(K_{n,\lfloor\beta n\rfloor})}=\exp\left(\beta+\frac{1}{\beta}-1\right).$$

\subsection{Stronger conjectures}

There are stronger conjectures than the one proved in this paper. Recall that $F_s(G)$ is the number of spanning forests of $G$ with $s$ components.

\begin{Conj} \label{Fs-conjecture}
Let $G$ be a simple connected graph on $n$ vertices. If $2\leqslant s\leqslant n$, then
$$\frac{F_s(G)}{T(G)}\geqslant \frac{F_s(K_n)}{T(K_n)}$$
with equality if and only if $G=K_n$.
\end{Conj}

The case $s=2$ is Lemma~\ref{2-forests-complete-graph}. An even stronger conjecture is the following.

\begin{Conj} \label{Fs-Fs-1-conjecture}
Let $G$ be a simple connected graph on $n$ vertices. Then for any $2\leqslant s \leqslant n$ we have
$$\frac{F_s(G)}{F_{s-1}(G)}\geqslant \frac{F_s(K_n)}{F_{s-1}(K_n)}$$
with equality if and only if $G=K_n$.
\end{Conj}

This conjecture would follow from the following one. Before we state it, we need a definition.

\begin{Def}
Let $H=(A,B,E)$ be a bipartite graph. For a set $S\subseteq A$ let $N_H(S)=\{u\in B\ |\ \exists v\in S: (u,v)\in E\}$, that is, the neighbors of $S$. The graph $H$ has the normalized matching property if
for every set $S\subseteq A$ we have
$$|N_H(S)|\geqslant \frac{|B|}{|A|}|S|.$$
\end{Def}

An alternative more symmetric formulation of the normalized matching property is that there exists a $w:E(H)\to \mathbb{R}_{\geq 0}$ such that
$$\sum_{e\in E: v\in e}w(e)=\begin{cases} \frac{1}{|A|} & \mbox{if}\ v\in A, \\ \frac{1}{|B|} & \mbox{if}\ v\in B. \end{cases}$$

\begin{Conj} \label{NMP-forests}
Let $s\geq 2$, and let $\mathcal{F}_{n,s}$ be the set of forests of $K_n$ with $s$ components. Let $H_{n,s}=(\mathcal{F}_{n,s},\mathcal{F}_{n,s-1},E_{n,s})$ be the bipartite graph with parts $\mathcal{F}_{n,s}$ and $\mathcal{F}_{n,s-1}$ with an edge between $F\in \mathcal{F}_{n,s},\ F'\in \mathcal{F}_{n,s-1}$ if $F\subseteq F'$. Then for every $2\leqslant s\leqslant n$ the bipartite graph $H_{n,s}$ has the normalized matching property.
\end{Conj}

The reason why Conjecture~\ref{NMP-forests} implies Conjecture~\ref{Fs-Fs-1-conjecture} is that we can take $S_G\subseteq \mathcal{F}_{n,s-1}$ to be the set of forests of $G$ with $s-1$ components. (Because of the symmetric formulation, it does not matter on which side we take a subset.) Then $N_{H_{n,s}}(S_G)$ is just the set of forests of $G$ with $s$ components. Applying the normalized matching property to $S_G$ immediately gives Conjecture~\ref{Fs-Fs-1-conjecture}. 

Interestingly, the statement that $H_{n,s}$ has the normalized matching property for every $2\leqslant s\leqslant n$ is equivalent to the property that the poset of forests has the so-called LYM-property: every antichain $\mathcal{A}$ of the forests of $K_n$ satisfies the inequality
$$\sum_{s=1}^n\frac{|\mathcal{A}\cap \mathcal{F}_{n,s}|}{|\mathcal{F}_{n,s}|}\leqslant 1.$$
This is false in general: the book graph with five triangles already fails the LYM property.

\bigskip

\noindent \textbf{Use of large language models.} We used ChatGPT 5.5 Pro and ChatGPT 5.6 Sol. Throughout the text, we indicate which model contributed each result. The authors take full
responsibility for the correctness of the arguments, citations, and the exposition of the paper. 

\bibliography{bibliography}
\bibliographystyle{plain}

\end{document}